\documentclass[12pt,english]{smfart}

\usepackage{smfthm,smfenum}
\usepackage{hyperref,smfhyperref}

\hypersetup{colorlinks,citecolor=blue,filecolor=blacklinkcolor=red, urlcolor=magenta,linkcolor=red}

\usepackage{csquotes}
\usepackage{amssymb}
\usepackage{bm}
\usepackage{mathrsfs}
\usepackage{tikz-cd,tikz}
\usepackage{stmaryrd}
\usepackage[normalem]{ulem} 
\theoremstyle{plain}

\newcommand{\bbR}{\mathbb{R}}

\newcommand{\dd}{\textrm{d}}

\title{Projective Differential Geometry and asymptotic analysis in General Relativity}
\alttitle{french}
\begin{abstract} {In this text, we explore the tools that Projective Differential Geometry can provide for the asymptotic analysis of classical fields on projectively compact manifolds. We emphasise on the case of order 2-compactifications and develop, in this case, an projective exterior differential tractor calculus that is analogous to that existing for conformally compact manifolds.  } \end{abstract}
\begin{altabstract} Dans cet article, nous explorons les outils que fournit la géométrie différentielle projective pour l'analyse asymptotique de champs classiques sur des variétés projectivement compactes. Nous portons une attention particulière au cas des compactifications d'ordre 2 et développons, dans ce cas, un calcul extérieur pour les tracteurs projectifs analogue à celui existant sur des variétés conformément compactes.     \end{altabstract}
\thanks{I am very thankful to my thesis advisor Jean-Philippe Nicolas, for his on-going support and innumerable discussions, as well as A. R. Gover and his students for their warm welcome during my visit to Auckland University at the end of 2019 which enabled me to considerably advance my knowledge on this fascinating topic. I would also like to acknowledge the funding received by ED MathSTIC and University of Western Brittany that made the trip possible. }
\author{Jack A. Borthwick}
\address{Univ. Brest,UMR CNRS 6205, Laboratoire de Mathématiques de Bretagne Atlantique, 6 av. Victor Le Gorgeu, CS 93837, 29238 BREST cedex 3}
\email{jack.borthwick@math.cnrs.fr}
\begin{document}
\maketitle

\section{Introduction}
Defining the asymptotics of a complete, non-compact, Lorentzian manifold $(M,g)$ can sometimes be achieved by a construction one could coin as \enquote{geometric compactification}. The idea is to weaken the geometric structure, at the price of some invariants, and seek to realise $M$ as the interior of a manifold with boundary $\overline{M}$ such that the weaker structure has an extension to boundary points. Thinking of a Lorentzian manifold as a Cartan geometry~\cite{Sharpe:1997aa} modeled on the Klein geometry $G/H$ with  $G=\bbR^n\rtimes SO_+(1,n)$, $H=SO_+(1,n)$, one can perhaps conceive the \enquote{weakening} of the geometry as shifting to a model with larger structure group $G$.

There are at least two natural ways in which one may weaken the structure of a Lorentzian manifold $(M,g)$ in the above sense. The first, and most prominent example, is to consider the geometric structure defined by the conformal class $\bm{c}=\{ \Omega^2 g, \Omega \neq 0\}$ of the metric $g$. The second, that defined by the \emph{projective} class $\bm{p}$ of its Levi-Civita connection $\nabla_g$ i.e. the class of all torsion-free affine connections on the tangent bundle $TM$ that have the same unparametrised geodesics as $\nabla_g$. If in either case one can extend the structure to some boundary $\partial M$, we refer to the resulting manifold as a conformal or projective compactification of $(M,g)$.

Given such a compactification, one can hope to extract information about the asymptotic behaviour of solutions of partial differential equations by studying their behaviour near the boundary. To this end, the geometric structure guiding the compactification procedure can be an important ally; in particular, the conformal and projective structures associated to the Riemannian geometry are examples of parabolic geometries~\cite{Cap:2009aa} on $M$ and come with tools that enable the construction of invariant differential operators - that act on appropriate objects - which are naturally well-defined at boundary points!

There is a vast literature on the application of these ideas in the case of conformal compactifications and a complete introduction to the topic can be found in~\cite{daude_hafner_nicolas_2018}. One should also mention work on the development of conformal scattering theories such as~\cite{Mason:2004wg,AIF_2016__66_3_1175_0}. In this article, we shall instead consider the case of \emph{projective compactifications}.

There are two main motivations to our interest in projective compactifications. The first is directly related to the nature of the information preserved by conformal geometry and the resulting classification its boundary points. Emphasis is put on the null structure (e.g. unparametrised null-geodesics), which is an invariant of the structure, contrary to timelike geodesics, which are not.\footnote{If $\gamma$ is a time-like geodesic for $g$, i.e. $g(\dot{\gamma},\dot{\gamma})< 0$ with the signature convention $(-,+\dots, +)$, then it will remain a time-like curve for any $\hat{g}\in\bm{c}$, however, it will generally no longer satisfy $\hat{\nabla}_{\dot{\gamma}} \dot{\gamma} = k\dot{\gamma},$ for any smooth function $k$.} This dissymmetry is reflected, for instance, in the conformal boundary of asymptotically flat space-times where time-like infinity is reduced to two points ($i_+, i_-$). Intuitively, this suggests that the obtained structure is perhaps not best-adapted to the study of massives fields for their information should accumulate asymptotically at these two points as they cannot propagate at the speed of light. 

On the other hand, if we consider the projective compactification of Minkowski space-time, represented in Figure~\ref{fig:pcminkowski}, we obtain a richer classification of boundary points and, in particular, time-like infinity is a true hypersurface of the compactified manifold.
\begin{figure}
\begin{center}
\includegraphics[scale=.45]{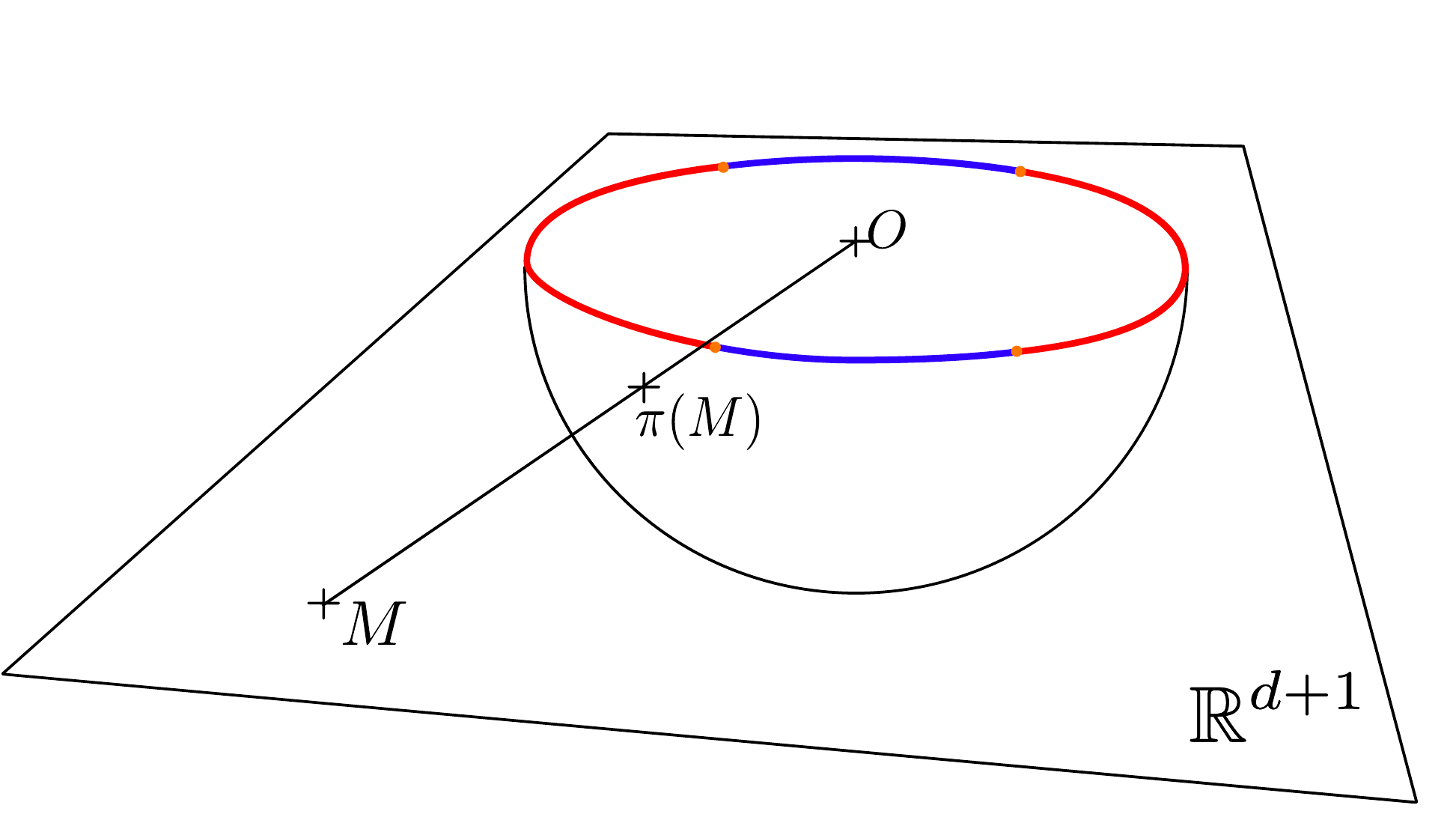}
\caption{\label{fig:pcminkowski} Projective compactification of Minkowski spacetime by central projection. The boundary at infinity is on the equator which splits into the three parts: in blue, time-like infinity, in red, space-like infinity, and, in orange, null infinity. \newline In order to recognise the geometric structure of Minkowski spacetime one must restrict the natural action of $SL_{d+2}(\bbR)$ on the projective sphere to the subgroup that preserves the one-form $e^*_{d+2}$ normal to the plane, and the degenerate metric on the dual space $(\bbR^{d+2})^*$ given by $H=-e^*_1\odot e^*_1 + \sum_{i=2}^{d+1}e^*_i\odot e^*_i.$ The decomposition of the equator is a description of the orbits of this action. }
\end{center}
\end{figure}

The second motivation that corroborates the hope we place in projective compactifications is a result due to L. Hörmander~\cite[Theorem 7.2.7]{Hormander:1997aa} with an implicit projective flavour. Hörmander derives an asymptotic expansion of solutions to the Klein-Gordon equation on Minkowski space time $M\cong \mathbb{R}^{1+d}$ in which the coefficients are only dependent on the projective parameter $\frac{x}{t}$. The proof relies on a decomposition of the field into positive and negative frequency parts achieved through the diagonalisation, in Fourier space, of the Klein-Gordon operator with domain $L^2(\mathbb{R}^{d})\times H^1(\mathbb{R}^{d})$ defined by:
\[\textrm{KG}=\left(\begin{array}{cc} 0 & \Delta-m^2 \\ 1 & 0\end{array} \right).\]
This \enquote{diagonalisation} reduces the study of the equation $\frac{\partial \psi}{\partial t} = \textrm{KG}\psi $ to that of two uncoupled equations involving the pseudo-differential operators $\pm i \sqrt{-\Delta + m^2}$. The asymptotic expansion then results from a precise study of an integral formula for the solutions to these equations.  Precisely, Hörmander shows that if $u \in S'(\mathbb{R}^{1+d})$ is a solution of : \[\left\{ \begin{array}{l}\partial_t u = i(-\Delta+1)^{\frac{1}{2}}u, \\ u(0,x)=\varphi(x), \varphi \in S(\mathbb{R}^d), \end{array} \right.\]
then $u(t,x) = U_0(t,x) + U_+(t,x)e^{\frac{i}{\rho}}$, with $U_0 \in S(\mathbb{R}^{1+d})$,  $\rho=(t^2-|x|^2)^{-\frac{1}{2}}$, and\footnote{$(+0+i\rho)^{d/2}=\displaystyle \lim_{\varepsilon \to 0^+} \exp(\frac{d}{2}\log(\varepsilon + i\rho))=\rho^{d/2}e^{id\pi/4}$ where $\log$ is the Principal complex logarithm.}:
\begin{equation} \label{eq:expansion_hormander}U_+(t,x) \sim (+0 +i\rho)^{\frac{d}{2}}\sum_0^\infty \rho^jw_j(t,x).\end{equation}
Setting  $\tilde{x}=\rho x$ and writing the Fourier transform of $\varphi$, $\hat{\varphi}$, we can rewrite the expression given for the first coefficient in the expansion as: 
\[w_0(t,x)=\left\{\begin{array}{ll} (2\pi)^{-d/2}\sqrt{1+|\tilde{x}|^2}\hat{\varphi}(-\tilde{x}) & \textrm{if $t^2 > |x|^2$},\\ 0 & \textrm{otherwise}. \end{array}\right. \] 
It is clear from this that $w_0$ is a function of $\frac{x}{t}$; Hörmander shows that this is also the case for $w_j, j>0$. 

An interesting point here is that the functions $(\rho, \tilde{x})$ introduced above define a coordinate system on the interior of the future lightcone $\{t>|x|\}$ that is compatible with the projective compactification.\footnote{They can equally serve to construct this part of the boundary from the inside.} Furthermore, the points $\{\rho=0\}$, correspond exactly to future time-like infinity. This motivates further the conjecture that the coefficients $w_j$ can be interpreted geometrically in terms of the compactified space. 

Another interesting remark can be made about the form of Hörmander's expansion given in Equation~\eqref{eq:expansion_hormander}; one may recognise an expression similar to the one would obtain, in a $1$-dimensional situation (half-line), by applying the Frobenius method~\cite{Ince:1956aa} to solve the equation \enquote{off} the boundary ($\rho=0$). In the case of conformally compact manifolds, \emph{formal} series of a similar nature have been obtained in the form of a boundary calculus in~\cite{Gover:2014aa,Gover:2015aa}. These series provide formal solutions to a Dirichlet type problem involving a natural second order differential operator on a conformally compact manifold known as $I\cdot D$ (cf. ~\cite[Chapter 3, \S3.9]{daude_hafner_nicolas_2018}). 

In the work presented here we consider the structures available in the case of a projectively compact Lorentzian manifold that can be used to provide parallels to the results in~\cite{Gover:2015aa}. The goal will be to construct projectively invariant differential operators that act on appropriate objects, known as (weighted) tractors, and generalise -- in a sense to be made precise -- operators that appear in equations from physics like the Klein-Gordon equation, and its positive integer spin version, the Proca equation. Tractors were first (re)-introduced in their present form in~\cite{Bailey:1994aa} reviving and building upon on the work of T. Thomas~\cite{Thomas:1934aa}; their basic theory is recalled briefly in Section~\ref{sec:densities_tractors}. 
These objects make sense at the boundary and the idea is to extract the desired asymptotic information from their traces.

Our original ambition was to obtain a geometric proof, via methods of projective differential geometry, of Hörmander’s result in Minkowski spacetime. However, it turned out that Minkowski spacetime and, more generally, (asymptotically) scalar flat projectively compact metrics, that we will relate to projective compactness of order $1$, are an exceptional case where a part of the structure degenerates, and it is not quite clear as to how one can overcome the obstructions this entails; this will be the object of future work.

This realisation lead us to study in deeper detail the non-scalar flat case (projective compactness of order $2$),  where a number of results already exist in conformal geometry and to question to what extent there are projective analogues. The main bulk of this work is presented in Section~\ref{sec:exterior_tractor} where a projective exterior tractor calculus is developed that enables us to obtain a formal solution operator like that in~\cite{Gover:2015aa} for the Proca equation.

The author’s understanding of the topic was significantly advanced during a trip to Auckland in New Zealand financed by the University of Western Brittany, the Brittany region and ED MathSTIC to whom he expresses once more his gratitude.

\medskip

The article is organised as follows: \begin{itemize}\item  in Section~\ref{sec:densities_tractors}, we outline the basic material, notations and definitions used in this work. We also discuss in detail the projective compactifications of Minkowski and de-Sitter spacetimes. \item in Section~\ref{sec:projective_laplace1} we discuss the construction of a projective D'Alembertian operator to englobe the Klein-Gordon equation and the shortcomings of scalar-flatness, \item in Section~\ref{sec:exterior_tractor}, we construct a projective exterior calculus that we will use to write a tractor version of the Proca equation, \item finally in Section~\ref{sec:asymptotics}, we discuss how these developments can be used in the asymptotic analysis of the Klein-Gordon and Proca equations.  \end{itemize}

Throughout this article, we will use Penrose's abstract index notation~\cite{Penrose:1984aa}. $M$ will denote a smooth manifold of dimension $n=d+1\geq 2$, and connections on the tangent bundle $TM$ will systematically be assumed torsion-free.

\section{The basic structures} \label{sec:densities_tractors}
\subsection{Densities and projective structures}
We recall that a projective density of weight $\omega \in \mathbb{C}$ is a section of the associated bundle $\mathcal{E}(\omega)=L(TM)\times_{GL_n(\bbR)}\bbR$ where the action of $GL_n(\bbR)$ on $\bbR$ is given by the group homomorphism $A \mapsto |\det A|^{\frac{\omega}{n+1}},$ and $L(TM)$ is the frame bundle.\footnote{The frame bundle can be defined as $L(TM)=\{ (x,u_x), x\in M, u_x \in GL(\bbR^n,T_xM)\}$.}
If $\mathcal{B}$ is any vector bundle over $M$, the tensor product: $\mathcal{B}\otimes \mathcal{E}(\omega)$ will be abbreviated to: $\mathcal{B}(\omega)$.

A projective structure on $\bm{p}$ on $M$ is an equivalence class of projectively equivalent torsion-free affine connections on the tangent bundle $TM$. Two torsion-free affine connections $\nabla$ and $\hat{\nabla}$ are projectively equivalent if they have the same unparametrised geodesics; this is true if and only if one can find a one-form $\Upsilon_a$, such that for any vector field $\xi^a$:
\begin{equation} \label{eq:weylprojclass} \hat{\nabla}_a \xi^b = \hat{\nabla}_a\xi^b +  \Upsilon_c\xi^c\delta_a^b + \Upsilon_a\xi^b.\end{equation}
This relation will be written: $\hat{\nabla} = \nabla + \Upsilon.$

It can be useful to translate Equation~\eqref{eq:weylprojclass} in terms of the local connection forms of a local frame $(e_j)_{i\in \llbracket1,n\rrbracket}$ of $TM$ with dual frame $(\omega^i)_{i\in \llbracket,1,n\rrbracket}$ as so:
\begin{equation}\label{eq:proj_equiv_local} \hat{\omega}^i_{\phantom{i}j} = \omega^{i}_{\phantom{i}j} + \Upsilon(e_j)\omega^i + \Upsilon\delta^i_{\phantom{i}j}. \end{equation}
From this one can deduce the relationship between the covariant derivatives induced by $\nabla$ and $\hat{\nabla}$ on all associated vector bundles to $L(TM)$.
In particular, if $\sigma$ is a section of $\mathcal{E}(\omega)$, then:
\begin{equation} \label{eq:change_connection_densities} \hat{\nabla}_a \sigma = \nabla_a \sigma + \omega \Upsilon_a \sigma. \end{equation}
Similarly if, $\mu_b$ is a section of $T^*M$ then:
\begin{equation} \label{eq:change_connection_forms} \hat{\nabla}_a \mu_b = \nabla_a \mu_b - 2\Upsilon_{(a}\mu_{b)}. \end{equation}

Projective densities play an important role throughout this work. Loosely, they can be seen to replace scalar fields. Weighted tensors -- and later tractors -- will also appear naturally in the sequel. In fact, in some circumstances, working with weighted tensors can be all that is needed to achieve projective invariance in a given equation. Two simples examples are the Killing equation and, of course, the geodesic equation. Combining Equation~\eqref{eq:change_connection_densities} with either Equation~\eqref{eq:weylprojclass} or~\eqref{eq:change_connection_forms} one can show that if $\mu\in \Gamma(T^*M(2))$ and $\xi \in \Gamma(TM(-2))$ then the following equations are projectively invariant, i.e. independent of the connection in the class:
\[ \begin{gathered} \nabla_{(a}\mu_{b)}=0, \quad \xi^a\nabla_a\xi_b=0.\end{gathered} \]

Projective densities can also serve to introduce a notion of \emph{scale} analogous to that which is naturally present in conformal geometry. For this we recall that a torsion-free affine connection on $TM$ is said to be \emph{special} if it preserves a nowhere vanishing density $\sigma$. For example, Levi-Civita connections are special. 

At fixed weight\footnote{Although here the weight is non-essential.}, the preserved density is uniquely determined up to a constant factor and will be called the scale determined by $\nabla$.  Not all connections in a projective class are special, but there is always\footnote{Under the usual assumptions on the topology of $M$, in particular, second countability.} a special connection within a class. Indeed, if $\nabla$ is any torsion free connection and $\sigma$ is a nowhere vanishing density of weight $\omega$, then, $\hat{\nabla}=\nabla - \frac{1}{\omega}\sigma^{-1}\nabla \sigma$ is projectively equivalent to $\nabla$ and preserves $\sigma$. $\hat{\nabla}$ will in turn be called the scale determined by $\sigma$. 
\begin{rema}
\label{rem:special_connections}
Note that two special connections are necessary related by a one-form $\Upsilon$ of the form $\frac{\nabla_a \rho}{\omega \rho}$ for some non-vanishing function $\rho$. Furthermore, if $\nabla$ is special and preserves $\sigma\in \Gamma(\mathcal{E}(\omega))$ then $\hat{\nabla}=\nabla + \frac{\nabla \rho}{\omega \rho}$ preserves the density $\hat{\sigma}=\rho^{-1}\sigma.$ 
\end{rema}
\subsection{Tractors and Cartan's normal connection}
By a result due to E. Cartan~\cite{Cartan:1924aa, Sharpe:1997aa}, a projective structure on $M$ in the above sense gives rise to a unique normal and torsion-free Cartan geometry on $M$ modeled on the Klein geometry $G/H$ where $G=SL_{n+1}(\bbR)$ and $H$ is the subgroup that preserves the ray directed by $(1,\dots, 0)$, ($\mathfrak{g}$ and $\mathfrak{h}$ will denote their respective Lie algebras.) That is, a $H$-principal bundle over $M$ and a $\mathfrak{g}$-valued one form, $\omega$ on $P$ that satisfies:
\begin{enumerate}
\item $\forall h\in H, R_h^*\omega= \textrm{Ad}_{h^{-1}}\omega$,
\item $\forall X \in \mathfrak{h}, \omega(X^*)=X$, where $X^*$ is the fundamental field defined by $(X^*)_p = \left.\frac{\dd}{\dd t} \left( p\exp(Xt) \right)\right\rvert_{t=0}, p \in P$.
\item For every $p \in P$, $\omega_p : T_pP \rightarrow \mathfrak{g}$ is an isomorphism.
\item \emph{(Torsion free)} The curvature form\footnote{The curvature form is the exterior covariant derivative of the connection form.}, $\Omega = D\omega=d\omega +\frac{1}{2}[\omega\wedge\omega]\footnote{If $\alpha,\beta$ are $\mathfrak{g}$-valued forms of degree $k$ and $l$, then: $\displaystyle [\alpha\wedge \beta](X_1,\dots, X_{k+l})=\frac{1}{k!l!}\sum_{\sigma \in \mathfrak{S}_{k+l}} [\alpha(X_{\sigma(1)},\dots X_{\sigma(k)}),\beta(X_{\sigma(k+1)},\dots,X_{\sigma(k+l)})]$.}$ is $\mathfrak{h}$-valued. 
\item \emph{(Normal)} The curvature function $K : P \rightarrow \textrm{Hom}(\Lambda^2(\mathfrak{g}/\mathfrak{h}),\mathfrak{h})$ takes its values in the kernel of the projective Ricci homomorphism (cf.~\cite{Sharpe:1997aa}).
\end{enumerate}

The Cartan connection $\omega$, whilst not a principal connection on the $H$-bundle $P$, naturally induces such a connection on the associated bundle $Q=P\times_{H}G$, where $H$ acts on $G$ by left multiplication. Consequently, it induces a linear connection $\nabla$ on any associated vector bundle to $Q$.

The standard tractor bundle is defined as $\mathcal{T}=Q\times_{G}\bbR^{n+1}$, where $G$ acts on $\bbR^{n+1}$ canonically. It has the interesting decomposition structure described by the following short exact sequence: 
\begin{equation} \label{eq:suite_tracteur}0 \longrightarrow \mathcal{E}(-1) \overset{X}{\longrightarrow} \mathcal{T} \overset{Z}{\longrightarrow} TM(-1)\longrightarrow 0.\end{equation}
 
One can right-split this short exact sequence via a choice of connection $\nabla$ in projective class $\bm{p}$. This provides a convenient column vector representation of sections $T$ of $\mathcal{T}$ -- tractors -- as follows :
\[ T \overset{\nabla}= \begin{pmatrix}\nu^a \\ \rho\end{pmatrix}, \]
If one changes the splitting connection:  $\hat{\nabla}=\nabla + \Upsilon$, then the components of the column vector representation transform according to:
\begin{equation} \label{eq:tractor_transfo}T \overset{\nabla}= \begin{pmatrix}\nu^{a}\\\rho\end{pmatrix} \overset{\hat{\nabla}}=\begin{pmatrix} \nu^a \\ \rho - \nu^b\Upsilon_b\end{pmatrix}.  \end{equation}

In expressions with abstract indices, we will use capital latin letters $A,B, \dots,$ to denote tractor indices. The maps $X$ and $Z$ can be thought of as sections $X^A$ and $Z_A^a$ of $\mathcal{T}(1)$ and $\mathcal{T}^*(-1)$, respectively. If $\nabla$ is any connection, then $Y_A$ and $W^A_a$ will denote the \emph{non-canonical} maps that split the sequence. The column vector representation can then be replaced by the more compact notation:
\[ T^A\overset{\nabla}=\rho X^A + \nu^aW^A_a.\]
Note that: 
\begin{equation} X^AY_A=1, \,\, Z_A^aW^A_b=\delta^a_b, \,\, Z_A^aX^A=0, \,\, W^A_aY_A=0.  \end{equation}
We insist on the fact that although nothing indicates it in the notation, the sections $Y_A$ and $W^A_a$ \emph{depend} on the connection $\nabla$ used to split the tractor bundle.

The action of the affine connection induced by the normal Cartan connection on $\mathcal{T}$, can be calculated in the splitting determined by a connection $\nabla$ according to the formula :
\begin{equation} \label{eq:tractor_connection} \nabla^\mathcal{T}_a T^B \overset{\nabla}= (\nabla_a \rho -P_{ac}\nu^c)X^B + (\nabla_a \nu^b+\delta_a^b\rho)W_b^B. \end{equation}
Here, $P_{ab}$ is the Projective Schouten tensor of the connection $\nabla$, defined by the following decomposition of its Riemann tensor\footnote{$R_{ab\phantom{c}d}^{\phantom{ab}c}X^d= 2\nabla_{[ab]}X^c$}: 
\begin{equation}\label{eq:decomposition_riemann}R_{ab\phantom{c}d}^{\phantom{ab}c}=W_{ab\phantom{c}d}^{\phantom{ab}c} + 2\delta^c_{[a}P_{b]d}+\beta_{ab}\delta^{c}_{\phantom{c}d},\end{equation}
The (projective) Weyl tensor $W_{ab\phantom{c}d}^{\phantom{ab}c}$ is trace-free and $P,\beta$ are uniquely determined in terms of the Ricci tensor\footnote{$R_{ab}=R_{ca\phantom{c}b}^{\phantom{ca}c}$.}:
\[ \begin{cases}(n-1)P_{ab} = R_{ab} + \beta_{ab} \\ \beta_{ab}=-\frac{2}{n+1}R_{[ab].}\end{cases}\]
Whilst the Weyl tensor is projectively invariant, we recall that neither the Schouten tensor $P_{ab}$ nor $\beta$ are. They instead change according to: 
\[ \begin{gathered} \hat{P}_{ab}=P_{ab}-\nabla_a\Upsilon_b +\Upsilon_a\Upsilon_b,  \\ \hat{\beta}_{ab}= \beta_{ab}+2\nabla_{[a}\Upsilon_{b]},\end{gathered}\]
where it is understood that $\hat{P}_{ab}$ and $\hat{\beta}_{ab}$ are the corresponding tensors for the connection $\hat{\nabla}=\nabla + \Upsilon$.
\begin{rema}
The tensor $\beta$ is directly related to the curvature of the density bundles, for any $\sigma \in \Gamma(\omega)$:
\[2\nabla_{[a}\nabla_{b]}\sigma = \omega \beta_{ab}\sigma. \]
Since all our connections on $TM$ are assumed torsion-free it vanishes for special connections, in fact:
\begin{lemm}
\label{lemme:connexion_speciale}
If $\nabla$ is a special connection, then its Ricci and Schouten tensors are symmetric and all density bundles are flat.
\end{lemm}
\end{rema}

\medskip

For some computations, after fixing a choice of connection $\nabla$ in the projective class, it will be convenient to equip all composite tensor bundles mixing tractor and tensor indices with a natural connection\footnote{This boils down to imposing the Leibniz rule.} that mixes the tractor connection $\nabla^\mathcal{T}$ and $\nabla$. This connection will abusively be written $\nabla$, and the contents of Equation~\eqref{eq:tractor_connection} can be re-written:
\begin{equation}
\begin{gathered}
\nabla_aY_A = P_{ab}Z_A^b,\quad \nabla_a Z_A^b = -\delta_{a}^b Y_A,
\\ \nabla_a X^A= W_a^A,\quad \nabla_a W^A_c = -P_{ac}X^A.
\end{gathered}
\end{equation}
\subsection{The Thomas $D$-operator} 
\label{sec:thomasd} The tractor connection can be used to define a projectively differential operator that plays a central role in our later developments known as the Thomas $D$-operator. It acts on any density or weighted tractor bundle as follows:
\[D_AF^\circ \overset{\nabla}= \omega Y_A F^\circ + \nabla_a F^\circ Z_A^a,\]
where the $\circ$ denotes any set of tractor indices. Projective invariance can be verified by direct computation.

The Thomas $D$-operator is closely analogous to a covariant derivative with tractor indices. In particular, it satisfies the Leibniz rule:
\[D_A(F^\circ G^\circ)=(D_AF^\circ)G^\circ + F^\circ(D_A G^\circ).\]
It is interesting to note that this is not true of the conformal Thomas D-operator. 

We also define here the \emph{weight operator}: $\bm{\omega}=X^AD_A$ that acts on an arbitrary weighted tractor bundle by: $F\mapsto \omega F$

\subsection{The metrisability equation}
\label{sec:me}
We have not yet mentioned the consequences of working with the projective class of the Levi-Civita connection of some metric $g$. In an arbitrary projective class $\bm{p}=[\nabla]$, it can in fact be that there is no such connection. The presence or not of one is governed by an over-determined projectively invariant equation known as the \emph{metrisability} equation~\cite{Eastwood:2008aa}:
\[\label{eq:metrisabilité} \nabla_c\sigma^{ab} -\frac{2}{n+1}\nabla_d\sigma^{d(a}\delta^{b)}_c=0, \tag{ME} \]
where the unknown is a section of $S^2TM(-2)$.
Any non-degenerate solution $\sigma^{ab}$ to the Metrisability equation produces a symmetric bilinear form $g=\sigma\sigma^{ab}$ whose Levi-Civita connection is in the projective class. Here : \[\sigma = \varepsilon^2_{a_1a_2\dots a_n b_1\dots b_n}\sigma^{a_1b_1}\dots\sigma^{a_nb_n},\] and $\varepsilon^2_{a_1a_2\dots a_n b_1\dots b_n}$ is the canonical\footnote{$\Lambda^nTM \otimes \Lambda^nTM$ is canonically oriented, so this map exists even if $M$ is not orientable, if $M$ is oriented, then we can consider the \enquote{square} of a volume form.} map $\Lambda^nTM \otimes \Lambda^nTM \rightarrow \mathscr{E}(2n+2)$. 

Starting from a pseudo-Riemannian manifold $(M,g)$, the metric $g$ therefore provides a canonical solution to~\eqref{eq:metrisabilité} given by: 
$\xi^{ab}=g^{ab}|\omega_g|^\frac{2}{n+1}$ where $|\omega_g|$ is the (positive) volume density.

The precise geometry of solutions to~\eqref{eq:metrisabilité} and the relationship with projective compactifications are studied in~\cite{Flood:2018aa}. One important fact is that applying a procedure known as \emph{prolongation}\footnote{For a general discussion on prolongation see~\cite{Branson:2006aa}. The specific case of the metrisability equation is treated in~\cite{Eastwood:2008aa}. An alternative proof is also given in~\cite{Cap:2014aa} using the projective tractor calculus. It is discussed there that the metrisability equation is an example of a (first) BGG equation, and the map $\sigma^{ab}\mapsto H^{AB}$ is an example of a BGG splitting operator.  }, to the above equation it can be shown that its solutions are in one-to-one correspondence with 2-tractors $H^{AB}$ that satisfy: \begin{equation}\label{eq:me2} \nabla_cH^{AB} +\frac{2}{n}X^{(A}\Omega^{\phantom{cE}B)}_{cE\phantom{B}F}H^{EF}=0, \tag{ME2}\end{equation} 
where: $\Omega^{\phantom{aB}C}_{aB\phantom{C}D}=Z^b_B\Omega_{ab\phantom{C}D}^{\phantom{ab}C}$ and $\Omega_{ab\phantom{C}D}^{\phantom{ab}C}$ is the curvature tractor defined by: $2\nabla_{[ab]}T^C=\Omega_{ab\phantom{C}D}^{\phantom{ab}C}T^D$.

A solution $H^{AB}$ of~\eqref{eq:me2} is said to be \emph{normal} if it is parallel for the tractor connection (i.e. the second term vanishes). These normal solutions are intimately related to Einstein metrics~\cite{Cap:2014aa,Cap:2016aa}.

\begin{rema} Applying the prolongation procedure to other overdetermined projectively invariant equations leads to similar correspondences. For instance, the case of the projective Killing equation is discussed in~\cite{Gover2018InvariantPO}.\end{rema}

\medskip
For future reference, we record here the expression of the tractor curvature in a given splitting determined by a connection $\nabla$:
 \begin{equation}
\label{eq:courbure_trac}
\Omega_{ab\phantom{C}D}^{\phantom{ab}C}\overset{\nabla}=-Y_{abd}X^CZ_{D}^d +W_{ab\phantom{c}d}^{\phantom{ab}c}W^C_cZ^d_D.
\end{equation}
In the above: $Y_{abc}=2\nabla_{[a}P_{b]c},$ is the projective Cotton tensor.\footnote{Interestingly enough, a version of the Cotton tensor has recently appeared in the Physics literature~\cite{PhysRevD.103.L121502}.}

\subsection{Projective compactifications}

\label{sec:compactification_projective}
The notion of projective compactness was introduced in~\cite{Cap:2014aa,Cap:2016aa}, from which we recall the definition:
\begin{defi}
Let $\bar{M}=M\cup \partial M$ be a smooth manifold with boundary, whose interior is $M$, and let $\nabla$ be an affine connection on $M$. A boundary defining function is a map $\rho$ that satisfies:
\begin{enumerate}
\item $\mathcal{Z}(\rho)=\{ x\in \bar{M}, \rho(x)=0\}=\partial M$,
\item $\textrm{d}\rho \neq 0$ on $\partial M$.
\end{enumerate} 
We will say that $\nabla$ is \emph{projectively compact of order $\alpha \in \mathbb{R}_+^*$} if for every point $x_0 \in \partial M$,  one can find a neighbourhood $U$ of $x_0$ in $\overline{M}$ and a boundary defining function $\rho$ on $U\cap M$ such that the connection : 
\begin{equation} \hat{\nabla} = \nabla + \frac{\textrm{d}\rho}{\alpha\rho}, \end{equation}
has a smooth extension to the boundary, i.e. for instance, the local connection forms of $\hat{\nabla}$, defined on $U\cap M$, in any frame $(e_i)$ on $U$ that is smooth up to the boundary, extend to $\partial M$.

By extension, a metric $g$ on $M$ is said to be projectively compact of order $\alpha$ if its Levi-Civita connection is projectively compact in the above sense. 
\end{defi}

\medskip

The definition is independent of the choice of defining function $\rho$, on the other hand, the parameter $\alpha$, cannot be removed. Thinking in terms of projectively compact metrics, it is related to volume asymptotics. 

It will be convenient to have a coordinate independent characterisation of boundary points. A \emph{boundary defining density} is a global section of $\sigma \in \mathcal{E}(\alpha)$ for a fixed weight $\alpha \in \mathbb{R}_+^*$ vanishing exactly on $\partial M$ and such that its expression in any local trivialisation on a neighbourhood of a boundary point $x_0 \in \partial M$ is a boundary defining function. Without loss of generality we can choose $\sigma >0$ on $M$\footnote{Note that the density bundle is canonically oriented.}.

We quote the following lemma that relates boundary defining densities and projectively compact connections of order $\alpha$:
\begin{lemm}[ {\cite[Proposition~2.3~\emph{(ii)}]{Cap:2016aa} }]
Let $\overline{M}$ be a manifold with boundary equipped with a projective structure $[\nabla]$ on the interior $M$ that extends to the boundary $\partial M$. Suppose that $\sigma \in \mathcal{E}(\alpha)$ is a boundary defining density and let $\hat{\nabla}$ be the scale determined by $\sigma$ on $M$, then: $\hat{\nabla}$ is projectively compact of order $\alpha$.
\end{lemm}
\medskip

As mentioned in the introduction, our interest will be in the case of projectively compact Lorentzian manifolds of order $\alpha \in \{1,2\}$. Understanding of the general geometry often comes from a good understanding of the model cases, which for our discussion will be the Minkowski ($\alpha=1$) and de-Sitter space-times ($\alpha=2$) that we shall now overview. Both can be identified with subsets of the projective sphere, to be thought of as the ray projectivisation\footnote{Quotient space of $\mathbb{R}^{n+1}\setminus\{0\}$ by the canonical action of $\bbR_+^*$, $(\lambda, x) \mapsto \lambda \cdot x$. It is an oriented version of projective space, without identification of antipodal points.} $P_+(\mathbb{R}^{n+1})\cong S^n$ of $\mathbb{R}^{n+1}\setminus \{0\}$, and their metric structure can be obtained by reducing the canonical projective structure on $S^n$, that is, the canonical action of $SL_{n+1}(\bbR)$ on $S^n$. The $2$-tractor, $H^{AB}$, coming from the metrisability equation, will play an important role in this. The reader may also note that our model examples are \emph{normal} solutions of the metrisability equation, we will not however systematically make this assumption in the general cases studied later. 

On the projective sphere, the tractor bundle is the trivial bundle $S^n\times \bbR^{n+1}$ and the tractor connection is equally trivial. Parallel tractors on the projective sphere can therefore be identified with constant vectors of $\bbR^{n+1}$. We can also make the following identifications in this simple case:
\begin{itemize}
\item Functions on $S^n$ are in one-to-one correspondence with functions on $\bbR^{n+1}\setminus \{0\}$ that are constant on each ray.
\item Densities of weight $\omega$ can be identified with homogenous functions $f$ on $\bbR^{n+1}\setminus\{0\}$ i.e. that satisfy: $f(tx)=t^{\omega}f(x), x\in \bbR^{n+1}\setminus\{0\}, t\in \bbR_+^*$. This follows from the fact that the frame bundle of densities with weight $1$ can be identified with $\bbR^{n+1}\setminus\{0\}$, for example through the map: 
\[ \begin{array}{lcl} \mathbb{R}^{n+1}\setminus\{0\} &\longrightarrow &S^n \times \mathbb{R}_+^* \\ (X_1,\dots,X_{n+1}) &\longmapsto & ([X_1,\dots,X_{n+1}], |X_1| + \dots + |X_{n+1}|),\end{array}\]
where, $[X_1,\dots, X_{n+1}]$ denote the homogenous coordinates of the ray.
\item The section $X^A$ in the short-exact sequence~\eqref{eq:suite_tracteur} can be understood as the factorisation of the map:
\[ X \in \mathbb{R}^{n+1}\setminus\{0\} \mapsto \rho(X)X \in \bbR^{n+1}. \]
\item Thinking of vectors on $\bbR^{n+1}$ as differential operators, the section $Z_A^a$ can be understood as the map that restricts differential operators on $\bbR^{n+1}$ to space of smooth $\bbR^*_+$-invariant functions. However, the result is not a vector field on $S^n$. Instead, $v(f)$ is a homogenous function on $\bbR^{n+1} \setminus \{0\}$ of weight $-1$, hence we have a weighted vector field with weight $-1$.
\end{itemize}

\subsubsection{Minkowski space-time}
\label{sec:minkowski_compactification}
Guided by our picture, we can recognise the affine plane within the projective sphere in the following manner. Consider the parallel tractor $I_A=(0,\dots,0,1)$ -- intuitively the normal the plane identified with a subset of an ambient $\bbR^{n+1}$ -- and restrict the natural action of $SL_{n+1}(\mathbb{R})$ to the subgroup that preserves $I_A$. The elements of this subgroup are easily seen to have the form:
\[\left(\begin{array}{cc} A & b \\ 0 & 1 \end{array} \right), A\in SL_{n}(\mathbb{R}), b\in \mathbb{R}^n;\]
which is the affine group.
The orbits of the restricted action split the sphere into 3 orbits, classified by the sign of the $1$-density $\sigma = I_AX^A \approx X_{n+1}$. Either $\sigma>0$ or $\sigma <0$ can be identified with the affine plane, say, for definiteness, $\sigma >0$, and the boundary sits at $\sigma=0$. $\sigma$ is therefore a boundary defining density.

To introduce the structure of Minkowski spacetime on the subset $\sigma >0$, one should reduce the projective group $SL_{n+1}(\mathbb{R})$ further by requiring that it equally preserve the constant tractor: $H=\textrm{diag}(-1,\dots, 1)$. This should be thought of as the solution to~\eqref{eq:me2} coming from the Minkowski metric, note in particular that $H^{AB}I_A=0$. The subgroup that preserves $H$ and $I$ can be represented as:
\[\left(\begin{array}{cc} A & b \\ 0 & 1 \end{array} \right), A\in SO(1,n-1), b\in \mathbb{R}^{n},\]
i.e. the Poincaré group $\mathbb{R}^n\rtimes SO(1,n-1)$.
The orbits of the restricted action decompose further: the boundary, $\sigma=0$, splits into orbits classified by the sign of $H(X,X)$. The three orbits can be interpreted as time-like $(H(X,X)<0)$, spacelike, ($H(X,X)>0$) and null infinity ($H(X,X)=0$). Note that each part of the boundary inherits structure, more precisely there is an induced metric on time-like and space-like infinity, and a conformal structure on null infinity.

It is worth mentioning an alternative construction of the compactification, from the inside out, based on local coordinates. 
\begin{figure}[h]
\begin{center}
\includegraphics[scale=0.55]{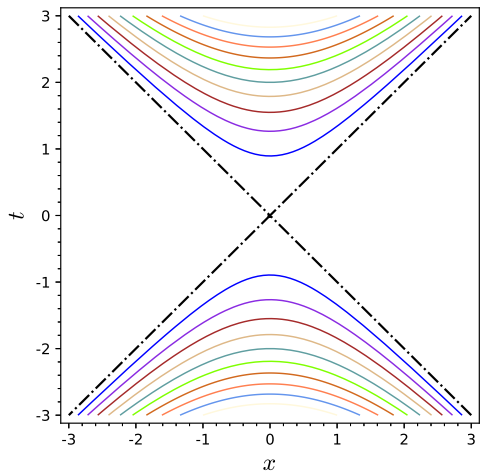}
\end{center}
\caption{\label{fig:fol}Foliation of the interior light-cone by hyperbolic sheets}
\end{figure}
 One could perhaps argue that this is how one has to proceed in a general setting. Let us concentrate first on future time-like infinity, and consider the hyperbolic foliation of the interior of the future light-cone of the origin, $\mathcal{S}^+$, in Figure~\ref{fig:fol}. It is given by the level sets of the function: \begin{equation}\label{eq:rho_minkowski} \rho=\frac{1}{\left\lvert x_0^2-\sum_{i=1}x_i^2\right\rvert}=\frac{1}{\sqrt{|g(x,x)|}}.\end{equation}
 
The coordinate chart $(\rho, \tilde{x}_1,\dots,\tilde{x}_d)$, $\tilde{x}_i=\rho x_i$, which is adapted to this foliation, identifies $\mathscr{S}^+$ with the open subset $]0,+\infty[\times \bbR^d$ of $[0,+\infty[$. Intuitively, $\{0\}\times \bbR^d$ should correspond to future time-like infinity, we only need to verify that the projective structure makes sense there. This can be accomplished directly by studying the local connection forms -- recorded in Appendix~\ref{annexe:minkowski}~-- of the Levi-Civita connection $\nabla$ and $\hat{\nabla}=\nabla +\frac{\nabla\rho}{\rho}$ in the local coordinate frame $(\frac{\partial}{\partial\rho}, (\frac{\partial}{\partial\tilde{x}_i})_{i\in \llbracket 1,d\rrbracket})$. Note however that appealing to~\cite[Theorem~2.6]{Cap:2016aa} projective compactness of order $1$ also follows from the form of the Minkowski metric in this chart:
\begin{equation} \label{eq:metric_coord_rho} g = \frac{\textrm{d}\rho^2}{\rho^4} -\frac{1}{\rho^2}\sum_{i,j}\underbrace{\left(\delta_{ij}-\frac{\tilde{x}_i\tilde{x}_j}{1+|\tilde{x}|^2} \right)}_{\rho^2g_{ij}}\textrm{d}\tilde{x}_i\textrm{d}\tilde{x}_j, \quad |\tilde{x}|=\sum_{i=1}^d \tilde{x}_i^2. \end{equation}
 
Using Equation~\eqref{eq:metric_coord_rho}, one easily identifies the metric on the boundary: \[h= \rho^2\left(g- \frac{d\rho^2}{\rho^4}\right)=\displaystyle \sum_{i,j}\left(\delta_{ij}-\frac{\tilde{x}_i\tilde{x}_j}{1+|\tilde{x}|^2} \right)\textrm{d}\tilde{x}_i\textrm{d}\tilde{x}_j.\]
Although this is not obvious from the outset, one can show directly that $\hat{\nabla}$ restricted to the boundary is the Levi-Civita connection for $h$.

Minor modifications of $\rho$ can be used to construct past time-like infinity and space-like infinity by a similar coordinate based method. On the other hand, projective null infinity requires a slightly different treatment. It can in fact be obtained by projectively compactifying the incomplete space-like and/or time-like projective infinity, which is projectively compact of order $2$.

We end this paragraph by making a brief link with $\rho$ and the boundary defining density $\sigma=X_{n+1}$ in our first description of the compactification. 
If we consider $\sigma_1$ the density on $S^n$ corresponding to the homogenous function:
\[f(X_1,\dots, X_{n+1})=\sqrt{X_1^2 -\sum_{i=2}^{n+1}X_i^2}  \]
Then: \[X_{n+1}=\left(\frac{X^2_1}{X_{n+1}^2}-\sum_{i=2}^{n+1}\frac{X_i^2}{X_{n+1}^2}\right)^{-\frac{1}{2}}f(X_1,\dots, X_n). \]
The functions $\frac{X_i}{X_{n+1}}=x_i$ define an affine coordinate chart on the region $\sigma >0$, and therefore we recognise $\rho$ as the boundary defining function defined by: $\sigma = f\sigma_1$.

\subsubsection{de-Sitter spacetime}\label{sec:compactification_de_sitter}
The other model case we consider is de-Sitter spacetime, a Lorentzian signature analogue of the Euclidean sphere, which can be described as the homogeneous space $SO(n,1)/SO(d,1)$\footnote{We recall that $n=d+1$}. It can also be identified with the hypersurface: \[\left\{ -X_1^2 + \sum_{i=2}^{n+1}X_i=k^2, X \in \mathbb{R}^{n+1}\right\},\] in $\mathbb{R}^{n+1}$. As with the sphere, the parameter $k \in \mathbb{R}$ is a scaling of curvature and has no importance for us, we will henceforth set $k=1$. 
To describe the projective compactification, introduce the non-degenerate $2$-tractor on $S^{n}$: $H^{AB}=\textrm{diag}(-1,1,\dots,1)$, and its inverse $H_{AB}$. We can interpret the equation defining the subset in $\mathbb{R}^{n+1}$ as an equation on the projective sphere as follows:
\[\sigma = H_{AB}X^AX^B >0.\]
Note that as a homogenous function of order $2$ on $\mathbb{R}^{n+1}$ it is a $2$-density on the projective sphere, and the image in $S^n$ of de-Sitter space by central projection is exactly the subset $\sigma >0$. In order to retrieve the structure of de-Sitter spacetime on this subset, we must restrict the action of $SL_{n+1}(\bbR)$ to the subgroup that preserves $H^{AB}$ as before.
The projective boundary is naturally defined by $\sigma=0$, however a striking difference with the compactification of Minkowski spacetime is that the action of $SO(n,1)$ on the projective sphere does not restrict to an action on the boundary $\sigma=0$ of de-Sitter space and, consequently, we do not get a projective structure on the boundary, instead it inherits a conformal structure.

We can also adopt a coordinate based approach: de-Sitter space time can be described as $\mathbb{R}_\psi \times S^d$ with the metric :
\begin{equation} \label{eq:desitter_metric} g = -\textrm{d}\psi^2 +\cosh^2\psi d\sigma^d,\end{equation}
where $d\sigma^d$ is the usual Euclidean metric of the unit $d$-sphere in $\mathbb{R}^{d+1}$.
Consider, for instance, the scalar field $\rho= \frac{1}{2\cosh^2\psi}$. In this case:
\[g= -\frac{\textrm{d}\rho^2}{4\rho^2} +\frac{1}{2\rho}\left(-\frac{\textrm{d}\rho^2}{1-2\rho} + \textrm{d}\sigma^{d} \right).\]
Since $h=-\frac{\textrm{d}\rho^2}{1-2\rho} + \textrm{d}\sigma^d$ extends smoothly to $s=0$,~\cite[Theorem 2.6]{Cap:2016aa} allows us to conclude immediately that de-Sitter space is projectively compact of order $2$. This fact can of course be verified directly upon inspection of the connection forms given in Appendix~\ref{annexe:desitter}.


\section{A projective D'Alembertian operator}\label{sec:projective_laplace1}
We have now laid out basic material we require to begin our main discussion. In this section $(M,g)$ is a $n=d+1$ dimensional Lorentzian manifold, that we shall assume projectively compact of order $\alpha\in \{1,2\}$. Its Levi-Civita connection shall be written $\nabla_g$ and is assumed not to extend smoothly to any point on the boundary. In this case we note that there is a canonical boundary defining density that we will use systematically:
\[ \sigma = |\omega_g|^{-\frac{\alpha}{n+1}} \in \Gamma(\mathcal{E}(\omega)).\]
where $|\omega_g|$ is the (positive) volume density. Note that $\zeta^{ab}=\sigma^{-\frac{2}{\alpha}}g^{ab}$ is a solution of the Metrisability equation~\eqref{eq:metrisabilité}. The corresponding tractor $H^{AB}$ can be expressed in the splitting determined by an arbitrary connection $\nabla$:
\[ H^{AB}\overset{\nabla}{=} \begin{pmatrix}\zeta^{ab}\\-\frac{\nabla_a\zeta^{ab}}{n+1}\\ \frac{P_{ab}\zeta^{ab}}{n}+\frac{\nabla_a\nabla_b\zeta^{ab}}{n(n+1)}\end{pmatrix}. \]
In particular in the Levi-Civita scale $\nabla_g$:
 \[ H^{AB}\overset{\nabla_g}{=} \begin{pmatrix}\zeta^{ab}\\0\\ \frac{P_{ab}\zeta^{ab}}{n}\end{pmatrix}. \]
From this we can see clearly that degeneracy of $H^{AB}$ is related to the scalar curvature and to its smooth extension to the boundary~\cite{CAP2015475}. Keeping in mind the work in~\cite{Flood:2018aa}, we shall assume that $H^{AB}$ is either of rank $n$, so $g$ is scalar-flat; this means that $\alpha=1$, or everywhere non-degenerate, which implies $\alpha=2$. 
\subsection{First attempt}
Our original interest in D'Alembertian-type operators that can act on tractors arose from the hope that they may provide a framework for the geometric interpretation of Hörmander's scattering result at time-like infinity for Klein-Gordon fields in Minkowski space-time discussed in the introduction. The D'Alembertian itself, $g^{ab}\nabla_a \nabla_b$ is not, of course, projectively invariant, acting on densities. In a first approach, it can be made projectively invariant by working with projective densities of arbitrary weight.

If $\sigma \in \Gamma(T^*M(\omega))$, then observe that $g^{ab}\nabla_a\sigma_b$ transforms under a change of connection $\hat{\nabla}=\nabla +\Upsilon$ according to:
\[ \begin{aligned} g^{ab}\hat{\nabla}_a \sigma_b &= g^{ab}\nabla_a\sigma_b + \omega \Upsilon^b\sigma_b - \Upsilon^b\sigma_b - \sigma_b \Upsilon^b \\
&=g^{ab}\nabla_a\sigma_b +(\omega-2)\Upsilon^b\sigma_b. \end{aligned}\]
It is therefore immediately invariant if $\omega=2$, however, we can avoid fixing the weight (that we hope to identify with a mass term) if we consider instead an operator of the form $\nabla_a + \zeta_a$, where the form $\zeta$ depends on the connection in the class and transforms according to:
\[\hat{\zeta}_a=\zeta_a-(\omega-2)\Upsilon_a.\]
It is possible to construct such a co-vector from any non-degenerate symmetric tensor $h_{ab}$, indeed, $\zeta_a=\frac{\omega-2}{n+3}h^{ac}\nabla_ch_{ab}$, is a suitable choice since:
\[\begin{aligned} h^{ac}\hat{\nabla}_c h_{ab} &= h^{ac}\nabla_c h_{ab} -2h^{ac}h_{ab}\Upsilon_c - h^{ac}h_{ac}\Upsilon_b -\Upsilon_a h^{ac}h_{cb}
\\&= h^{ac}\nabla_ch_{ab} -2\Upsilon_b -n\Upsilon_b -\Upsilon_b \\&= h^{ac}\nabla_c h_{ab} - (n+3)\Upsilon_b.\end{aligned} \]
With any such choice of $\zeta$, the quantity:
\[g^{ab}(\hat{\nabla}_a +\zeta_a)\sigma_b,\]
is projectively invariant.
Studying similarly the transformation rule for $\nabla_a\tau$ we find that for any $\tau\in \mathcal{E}(\omega)$, $\xi_a=h^{ac}\nabla_ch_{ab}$, $h\in S^2(T^*M)$, non-degenerate:
\[g^{ab}\left(\nabla_a + \frac{\omega-2}{n+3}\xi_a\right)\left(\nabla_b + \frac{\omega}{n+3}\xi_b\right)\tau,\]
is also projectively invariant. The operator can be extended to weighted tractors simply by coupling with the tractor connection.

Since we already have a metric $g$ at our disposal, it seems natural to set $h=g$ in the above. The resulting operator is a candidate for a D'Alembertian type operator, however, in the Levi-Civita scale the term $\xi_a$ vanishes and does not, therefore, provide a mass term...

Restricting the problem to Minkowski spacetime, we can try to solve the mass issue by exploiting some of the freedom left in the construction outlined above. Let $\rho$ be the boundary defining function defined by~Equation~\eqref{eq:rho_minkowski} on the future region of the future light cone of the origin of Minkowski spacetime. We note immediately that, with the Levi-Civita connection:
\[\Box\rho = -(n-3)\rho^3, \quad \nabla^a \rho \nabla_a \rho=\rho^4.\]
Set $h_{ab}=f(\rho)g_{ab}$, so that the new $\tilde{\xi}_a$ is given by:
\[\tilde{\xi}_a=\xi_a + (f(\rho))^{-1}f'(\rho)\nabla_a \rho.\]
Studying the form of $\nabla^a\tilde{\xi}_a$ and $\tilde{\xi}^a\xi_a$, it transpires that an interesting choice for $f$ is $f(\rho)=e^{-\frac{\alpha}{\rho}}$, for some $\alpha\in \mathbb{C}.$ For such a choice, expressed in the Levi-Civita scale:
\[\tilde{\xi}_a=\alpha \frac{\nabla_a\rho}{\rho^2}, \quad \tilde{\xi}^a\tilde{\xi}_a=\alpha^2, \quad \nabla^a\tilde{\xi}_a=-\alpha(n-1)\rho,\]
so that, if we write: $\displaystyle \beth \tau = g^{ab}\left(\nabla_a + \frac{\omega-2}{n+3}\tilde{\xi}_a\right)\left(\nabla_b + \frac{\omega}{n+3}\tilde{\xi}_b\right)\tau$, then in the Levi-Civita scale:
\[ \beth\tau= g^{ab}\nabla_a\nabla_b\tau -\alpha(n-1)\rho\omega\tau +\frac{2(\omega-1)}{n+3}g^{ab}\xi_a\nabla_b\tau + \frac{(\omega-2)\omega\alpha^2}{(n+3)^2}\tau. \]
Setting $\omega=1$ rids us of the first order term and setting $\alpha=im(n+3)$ we arrive, again in the Levi-Civita scale, at:
\[ \beth \tau + im(n-1)\rho\tau = (\Box + m^2)\tau.  \]
This has apparently given us what we sought; but perhaps in a trivial way. Indeed, there is no guarantee here that $\beth$ has an interesting extension to the boundary (or is defined there), and, our development depends in a non-trivial way on the boundary defining function, which is inherently dissatisfying. Finally, since $\tau$ is of weight $1$, and Minkowski spacetime is projectively compact of order $1$, then $\tau=\phi \sigma$, where $\sigma\in\mathcal{E}(1)$ is a boundary defining density. Hence, if $\phi$ extends to the boundary, so does $\tau$ but with vanishing boundary value. Working out the action of the operator on the component of $\tau$ expressed in a trivialisation that is valid up to the boundary, we find unfortunately that our concerns are founded as it has smooth coefficients that tend to $0$ at the boundary and reduces there to multiplication by $m^2$.

\subsection{A more natural operator}
The material introduced in Section~\ref{sec:densities_tractors}, shows that the metric projective structure provides us with a natural projective D'Alembertian constructed as follows: 
\[ \Delta^\mathcal{T} =H^{AB}D_AD_B.\]
During the author's trip to Auckland University, A.R. Gover suggested to the author that this operator would likely play a key role and may be a more successful candidate. We also note that it has already appeared in the literature~\cite{GOVER201820}. Of course, the results in Section~\ref{sec:me} corroborate this, given the geometric significance of $H^{AB}$. This is also a closer analogue to the operator \enquote{$I\cdot D$} in conformal tractor calculus~\cite[Chapter 3, \S3.9]{daude_hafner_nicolas_2018} than the previous attempt. 

The analogy with the conformal case is in fact complete in the case where $H^{AB}$ is non-degenerate. In the early stages of his on-going thesis work~\cite{smth}, Samuel Porath, a student of R. Gover, developed a boundary calculus in this case, that is in all points analogous to the results in~\cite{Gover:2014aa}. In fact, the important point is that $\Delta^\mathcal{T}$ defined above is part of an $\mathfrak{sl}_2$ algebra.
\begin{prop}[S. Porath]
\label{lemme:boundary_calculus}
Suppose $(M,g)$ is projectively compact of order $\alpha=2$, $H^{AB}$ non-degenerate, and let:
\begin{itemize} \item $x$ be the operator of multiplication by a boundary defining density $\sigma$, 
\item $y=-\frac{1}{\sigma^{-1}I^2}\Delta^\mathcal{T}$, with $I^2=H^{AB}D_A\sigma D_B\sigma$, \item $h=\bm{\omega}+\frac{d+2}{2}$. \end{itemize} Then $x,y,h$ form an $\mathfrak{sl}_2$-triple i.e.
\[[x,y]=h, \quad [h,x]=2x,\quad [h,y]=-2y. \]
\end{prop}
It is interesting to note that Proposition~\ref{lemme:boundary_calculus} is not specific to normal solutions to the Metrisability equation. 
The consequence of this that interests to us, is that following the same procedure as in~\cite{Gover:2014aa}, S. Porath developed a Boundary Calculus from which follows a formal solution operator that we relate to the asymptotics of solutions, we shall defer this discussion to Section~\ref{sec:operateur_solution_formelle}, where we will also follow~\cite{Gover:2014aa} to develop a boundary calculus for the Proca equation.

At this point we shall discuss why scalar-flatness is an obstruction in this context. First let us examine how $\Delta^\mathcal{T}$ acts on weighted densities. Let $f \in \Gamma(\mathcal{E}(\omega))$ and $\nabla \in \bm{p}$, then:
\[D_AD_B f \overset{\nabla}{=}(\omega-1)\omega f Y_AY_B + 2(\omega-1)\nabla_b f Y_{(A}Z_{B)}^b +(\nabla_a\nabla_b f +\omega P_{ab}f)Z_A^aZ_B^b.\]
Hence, writing: \[H^{AB}\overset{\nabla}{=} \zeta^{ab}W_a^AW_b^B -2\frac{\nabla_a \zeta^{ab}}{n+1}X^{(A}W^{B)}_b+\left(\frac{P_{ab}\zeta^{ab}}{n} + \frac{\nabla_a \nabla_b\zeta^{ab}}{n(n+1)}\right)X^AX^B,\] we find that: 
\[\begin{split}\Delta^\mathcal{T}f \overset{\nabla}{=} \omega(\omega-1)\left(\frac{P_{ab}\zeta^{ab}}{n} + \frac{\nabla_a \nabla_b\zeta^{ab}}{n(n+1)}\right)f + \zeta^{ab}(\nabla_a\nabla_b f +\omega P_{ab}f) \\-2\frac{\omega-1}{n+1}\nabla_a\zeta^{ab}\nabla_b f.\end{split}\]
In the scale $\nabla_\zeta$ the expression reduces to~:
\begin{equation} \label{eq:laplac} \Delta^\mathcal{T}f \overset{\nabla_\zeta}{=} \frac{\omega(\omega+n-1)P_{ab}\zeta^{ab}}{n}f + \zeta^{ab}\nabla_a\nabla_b f. \end{equation}
This indicates that in the case where the density $P_{ab}\zeta^{ab}$ is parallel for $\nabla_\zeta$, $\Delta^\mathcal{T}$ is a projectively invariant generalisation of the Klein-Gordon operator, with the proviso that the order-$0$ term be identified with the mass. Unfortunately, in the case of scalar-flat metrics like Minkowski spacetime, the term vanishes altogether and we have but a projective wave operator. Scalar-flatness is also an obstruction to our next developments, as well as Proposition~\ref{lemme:boundary_calculus}.

To see this observe first that the above formulae generalise to the case where $f$ is a weighted tractor by coupling a connection on $M$ with the tractor connection.
Consider now as in Proposition~\ref{lemme:boundary_calculus}, $x$, the operator acting on weighted tractors that multiplies by $\sigma$ and define the weight $\alpha-1$ co-tractor $I_A=D_A\sigma$, then~:
\begin{lemm}
\label{lemme:commutator_laplace_sigma}
\[ [x,\Delta^{\mathcal{T}}]=-\frac{\sigma^{-1}I^2}{\alpha}(2\bm{\omega}+d+\alpha),\]
where: $I^2=H^{AB}I_AI_B$ and $\bm{\omega}=X^AD_A$ is the weight operator. 
\end{lemm}
\begin{proof}
In the scale $\nabla_g$, $\sigma$ is parallel, so it commutes with $\nabla_g$. However, it does not commute with the weight operator as it increases weight by $\alpha$. Hence, if $F$ is an arbitrary tractor of weight $\omega$ then~:
 \[[x,\Delta^{\mathcal{T}}]F = (\omega(\omega+d)-(\omega+\alpha)(\omega+\alpha+d)\frac{P_{ab}\zeta^{ab}}{d+1}\sigma F.\]
 Again, in the scale $\nabla_g$, $I_A=\alpha \sigma Y_A$ and $I^2=\alpha^2\sigma^2\frac{P_{ab}\zeta^{ab}}{d+1}$ and the result ensues.
\end{proof}
If $g$ is scalar-flat $(\alpha =1)$ then on $M$, $P_{ab}\zeta^{ab}=0$ and $I^2=0$. Furthermore, $I_A$ is parallel for the tractor connection and extends naturally to $\overline{M}$, hence $I^2$ also extends smoothly to 0 on $\overline{M}$. So $[x,\Delta^{\mathcal{T}}]=0$ in this case. On the other hand, if $H^{AB}$ is non-degenerate on $M$ ($\alpha=2$), the function $\sigma^{-1}I^2$ is non-vanishing on $\bar{M}$ and, defining $y=-\frac{1}{\sigma^{-1}I^2}\Delta^{\mathcal{T}}$, we have reproduced Proposition~\ref{lemme:boundary_calculus}.

We see here directly the unfortunate consequences of scalar-flatness, Proposition~\ref{lemme:boundary_calculus} cannot hold because $x$ and $\Delta^{\mathcal{T}}$ commute and generate a trivial sub-Lie algebra. 

Let us remark that the absence of mass cannot be remedied by simply adding a constant of the type $m^2\sigma^{-\frac{2}{\alpha}}$ to the equation as this term will not behave well at the boundary. However, perhaps some of the ideas used to get a mass term in our first attempt can be adapted to accomplish it here.

\section{Exterior tractor calculus}\label{sec:exterior_tractor}

In this section, we will enrich the boundary calculus in Lemma~\ref{lemme:boundary_calculus} for tractor forms in order to develop similar projective methods for Proca style equations on $k$-forms. We return to the general setting of a connected $n$-dimensional projective manifold $(M,\bm{p})$. The first stage is to understand the exterior algebra of projective co-tractor $k$-forms, we begin by describing how the splitting of the exact sequence in Equation~\eqref{eq:suite_tracteur} induces a splitting of $\Lambda^k \mathcal{T}^*$:
\begin{lemm}
\label{lemme:tractorformcouple}
Let $(M, \bm{p})$ be a projective manifold of dimension $n$, $k\in\llbracket 1,n+1 \rrbracket$ and $\nabla \in \bm{p}$ then:
$$\Lambda^k \mathcal T^* \overset{\nabla}\cong (\Lambda^{k-1} T^*M)(k) \oplus (\Lambda^kT^*M)(k).$$
Any section $F_{A_1 \dots A_k}$ can be expressed as:
\begin{equation}F_{A_1 \dots A_k} =  \begin{pmatrix} \mu_{a_2 \dots a_k} \\ \xi_{a_1 \dots a_k} \end{pmatrix}=k\mu_{a_2 \dots a_k}Y_{[A_1}Z^{a_2}_{A_2}\cdots Z^{a_k}_{A_k]}+ \xi_{a_1 \dots a_k}Z^{a_1}_{A_1}Z^{a_2}_{A_2}\cdots Z^{a_k}_{A_k}.\end{equation}
(The second component vanishes if $k=n+1$).
Under the change of connection $\hat{\nabla}=\nabla + \Upsilon$ the components transform according to:
\begin{equation}
\label{eq:tractor_form_transform}
\left\{ \begin{array}{l} \hat{\mu}=\mu, \\ \hat{\xi}=\xi + \Upsilon \wedge \mu. \end{array} \right.
\end{equation}
\end{lemm}
The reader will find a proof of Lemma~\ref{lemme:tractorformcouple} in Appendix~\ref{app:preuve_lemme_tractorformcouple}. 

\subsection{Wedge product and exterior derivative}
The next stage is to describe how the usual operations of exterior calculus work with respect to the representation in Lemma~\ref{lemme:tractorformcouple}. The wedge product is relatively simple:
\begin{lemm}
\label{lemme:tractor_wedge}
Let $F\in \Lambda^k \mathcal{T}^*$, $G\in \Lambda^l \mathcal{T}^*$, and $\nabla \in \bm{p}$ on a projective manifold $(M,\bm{p})$. Suppose that:
$$F\overset{\nabla}=\left( \begin{array}{c} \mu \\ \xi \end{array}\right), \quad G\overset{\nabla}=\left(\begin{array}{c} \nu \\ \eta \end{array}\right),$$
then: 
\begin{equation} 
\label{eq:tractor_wedge}
F\wedge G \overset{\nabla}= \left(\begin{array}{c} \mu \wedge \eta +(-1)^k\xi\wedge \nu \\ \xi\wedge\eta \end{array} \right). \end{equation}
\end{lemm}

A tractor analogue of the exterior derivative is, as for the D'Alembertian, provided, by the Thomas $D$-operator (see Section~\ref{sec:thomasd}). The result can be stated as follows:
\begin{prop}
Let $D_A$ denote the projective Thomas $D$-operator then one can define a co-chain complex:
\[ \cdots \longrightarrow \mathcal{E}_{[A_1,\dots, A_k]}(\omega) \overset{\mathscr{D}}\longrightarrow \mathcal{E}_{[A_1,\dots, A_{k+1}]}(\omega-1)\overset{\mathscr{D}}\longrightarrow \mathcal{E}_{[A_1,\dots,A_{k+2}]}(\omega-2) \longrightarrow \cdots  \]
The operator $\mathscr{D}$ is defined on a section $F\in \mathcal{E}_{[A_1,\dots,A_k]}(\omega)$ by $$\mathscr{D}F= (k+1)D_{[A_1}F_{A_2\cdots A_{k+1}]}.$$
Furthermore, in terms of Lemma~\ref{lemme:tractorformcouple}, if $F\overset{\nabla}= \begin{pmatrix}\mu_{a_2\cdots a_k} \\ \xi_{a_1\cdots a_k} \end{pmatrix}$ then: 
\begin{equation} \label{eq:tractor_diff_ext_def} \mathscr{D}F \overset{\nabla}=\begin{pmatrix} (\omega+k)\xi_{a_2\cdots a_{k+1}} - k\nabla_{[a_2}\mu_{a_3 \cdots a_{k+1}]} \\ (k+1)\nabla_{[a_1}\xi_{a_2\cdots a_{k+1}]} + \frac{(k+1)!}{(k-1)!}P_{[a_1a_2}\mu_{a_3\cdots a_{k+1}]} \end{pmatrix}.\end{equation}
\end{prop}

\begin{proof}
First, we prove the expression for $\mathscr{D}F$ in the splitting associated with some connexion $\nabla \in \bm{p}$. Let $F\in\Lambda^k\mathcal{T}^*(\omega)$ be such that:  $$F_{A_1A_2\dots A_k}\overset{\nabla}= k\mu_{a_2 \dots a_k}Y_{[A_1}Z^{a_2}_{A_2}\dots Z^{a_k}_{A_k]} +\xi_{a_1\dots a_k} Z^{a_1}_{A_1}Z^{a_1}_{A_2}\dots Z^{a_k}_{A_k}.$$
By definition:
\[D_A F_{A_1A_2\dots A_k}= \omega F_{A_1A_2 \dots A_k}Y_A + Z_A^a\nabla_aF_{A_1A_2\dots A_k}. \]
Let us first concentrate on $\nabla_aF_{A_1A_2\dots A_k}$. Using Equation~\eqref{eq:tractor_connection}, we find that:
\[\begin{aligned}\nabla_aF_{A_1A_2\dots A_k} &=k \nabla_a\mu_{a_2\dots a_k}Y_{[A_1}Z_{A_2}^{a_2}\dots Z_{A_k]}^{a_k}+kP_{a[a_1}\mu_{a_2\dots a_k]}Z_{A_1}^{a_1}\dots Z_{A_k}^{a_k} \\&\quad+\nabla_a \xi_{a_1\dots a_k}Z_{A_1}^{a_1}\dots Z_{A_k}^{a_k}  \\&\quad- \xi_{a_1\dots a_k}\sum_{i=1}^k Z_{A_1}^{a_1}\dots Z_{A_{i-1}}^{a_{i-1}}\delta^{a_i}_aY_{A_i}Z_{A_{i+1}}^{a_{i+1}}\dots Z^{a_k}_{A_k}, \end{aligned}\]
in which the last term simplifies to: \[-k\xi_{aa_2\dots a_k}Y_{[A_1}Z_{A_2}^{a_2}\dots Z_{A_k]}^{a_k}.\]
Hence, in column vector form this can be written as: 
\[\nabla_a F = \begin{pmatrix} \nabla_a\mu_{a_2\dots a_k} -\xi_{aa_2\dots a_k} \\ \nabla_a \xi_{a_1\dots a_k} + kP_{a[a_1}\mu_{a_2\dots a_k]} \end{pmatrix}.\]
Since $\mathscr{D}F=(k+1)D_{[A}F_{A_1A_2\dots A_k]}$, any terms containing two $Y_{A_i}$ will not contribute in the final expression, furthermore for an arbitrary (weighted) tensor $T_{a_1\dots a_k}$:
\[T_{a_1\dots a_k}Z_{[A_1}^{a_1}\dots Z_{A_k]}^{a_k}=T_{[a_1\dots a_k]}Z_{A_1}^{a_1}\dots Z_{A_k}^{a_k}.\]
Hence: 
\[\begin{aligned}\mathscr{D}F \overset{\nabla}=& (k+1)\omega \xi_{a_1\dots a_k}Y_{[A}Z^{a_1}_{A_1}\dots Z^{a_k}_{A_k]}\\&+\uwave{k(k+1)(\nabla_a\mu_{a_2\dots a_k}-\xi_{aa_2\dots a_k})Y_{[A_1}Z^a_{A}Z_{A_2}^{a_2}\dots Z_{A_k]}^{a_k}} \\&+\left((k+1)\nabla_a\xi_{a_1\dots a_k} +k(k+1)P_{[aa_1}\mu_{a_2\dots a_k]}\right)Z^a_AZ^{a_1}_{A_1}\dots Z^{a_k}_{A_k}.\end{aligned}\]
Swapping $A$ and $A_1$ in the underlined term in order to respect the sign conventions laid out implicitly in Lemma~\ref{lemme:tractorformcouple} we arrive at the desired result.
We now proceed to calculate $\mathscr{D}^2F$ using Equation~\eqref{eq:tractor_diff_ext_def}. For readability, we treat each slot in the column vector notation separately. First of all, in the top slot we have:
\begin{equation} 
\label{eq:d2nul_top}
\begin{split}
(\omega+k)(k+1)\nabla_{[a_2}\xi_{a_3\cdots a_{k+2}]}+(\omega +k)\frac{(k+1)!}{(k-1)!}P_{[a_2a_3}\mu_{a_4\dots a_{k+2}]}\\-(k+1)(\omega+k)\nabla_{[a_2}\xi_{a_3\dots a_{k+2}]}-k(k+1)\nabla_{[a_2}\nabla_{a_3}\mu_{a_4\dots a_{k+2}]}\\ =(\omega +k)\frac{(k+1)!}{(k-1)!}P_{[a_2a_3}\mu_{a_4\dots a_{k+2}]}-k(k+1)\nabla_{[a_2}\nabla_{a_3}\mu_{a_4\dots a_{k+2}]}.
\end{split} 
\end{equation}
As for the bottom slot, we have:
\begin{equation}
\label{eq:d2nul_bottom}
\begin{split}
(k+2)(k+1)\nabla_{[a_1}\nabla_{a_2}\xi_{a_3\dots a_{k+2}]} +\frac{(k+2)!}{(k-1)!}\nabla_{[a_1}P_{a_2a_3}\mu_{a_4\cdots a_{k+2}]} \\ +\frac{(k+2)!}{k!}(\omega+k)P_{[a_1a_2}\xi_{a_3\cdots{a_{k+2}}]}-\frac{(k+2)!}{(k-1)!}P_{[a_1a_2}\nabla_{a_3}\mu_{a_4\dots a_{k+2}]}\\
=(k+2)(k+1)\nabla_{[a_1}\nabla_{a_2}\xi_{a_3\dots a_{k+2}]}+\frac{(k+2)!}{k!}(\omega+k)P_{[a_1a_2}\xi_{a_3\cdots{a_{k+2}}]} \\ +\frac{(k+2)!}{(k-1)!}Y_{[a_1a_2a_3}\mu_{a_4\cdots a_{k+2}]}. \end{split}
\end{equation}
Where we recall that $Y_{abc}:= 2\nabla_{[a}P_{b]c}$. Expressions~\eqref{eq:d2nul_top} and~\eqref{eq:d2nul_bottom} simplify enormously in the case that $\nabla$ is a special connection. Since there is always a special connection in any projective class\footnote{We assume that a manifold's topology is second-countable.}, there is no loss in generality if we restrict to this case. We appeal now to Lemma~\ref{lemme:connexion_speciale}, which states that $P_{ab}$ is symmetric so that:
\begin{equation}
\mathscr{D}^2 F \overset{\nabla}= \begin{pmatrix} -\dd^2 \mu \\ \dd^2\xi \end{pmatrix}.
\end{equation}
$d$ denotes here the covariant exterior derivative on the weighted tensor bundles. In general, $d^2\neq 0$, however, here, using again Lemma~\ref{lemme:connexion_speciale}, the density bundles are flat so we can conclude that:
$$ \mathscr{D}^2 F =0. $$
\end{proof}

\begin{rema}
In the preceding proof, one can avoid choosing a particular scale and directly use the fact that $2P_{[ab]}=-\beta_{ab}$ to show that the final expressions in equations~\eqref{eq:d2nul_top} and~\eqref{eq:d2nul_bottom} vanish.
\end{rema}

\subsection{Tractor Hodge duality derived from a solution of the metrisability equation}
Towards our aim to formulate a tractor version of the Proca equation, we describe here how one can use a solution to the metrisability equation to define a tractor Hodge star operator. The setting is as follows, we suppose we have a projective manifold with boundary $(\overline{M}, \bm{p})$, with oriented interior $M$ and boundary $\partial M$, in addition to a solution $\zeta$ of the metrisability equation with degeneracy locus $D(\zeta)=\partial M$ and such that $H^{AB}=L(\zeta^{ab})$ is non-degenerate on $\overline{M}$.

On $M$, $\zeta^{ab}$ is non-degenerate and defines a smooth metric on the weighted cotangent bundle $T^*M(1)$. The orientation on $M$ induces a natural orientation on $(\Lambda^nT^*M)(n)$, so it makes sense to talk of the positive volume form (induced by $\zeta$) $\omega \in \Gamma\left((\Lambda^nT^*M)(n)\right)$. To define an orientation on the tractor bundle, we introduce:  $\sigma = |\omega|^{-2} \in \Gamma(\mathcal{E}(2))$;  it is a positive defining density for the boundary. Set $I_A= D_A\sigma \in \mathcal{E}_A(1)$, since $H^{AB}$ is non-degenerate, the smooth function $\sigma^{-1}I^2$ is non-vanishing on $\overline{M}$, and so one can define on $M$:
$$J^0_B = \frac{\sigma^{-\frac{1}{2}}}{\sqrt{|\sigma^{-1}I^2|}}I_B.$$
Since any positively oriented orthonormal frame $(\omega^i_a)$ induces an orthonormal family of tractors $J^i_B=Z_B^b\omega^i_b$, we define an orientation of the tractor bundle by declaring that $J^0\wedge\dots\wedge J^n$ is positive. Since $(J^0, \dots, J^n)$ is an orthonormal family with respect to $H^{AB}$, this procedure also yields a local expression of the positive tractor volume form in the splitting of the Levi-Civita connexion $\nabla_g$ of $g^{ab}=\sigma\zeta^{ab}$ on $M$: 
$$\Omega_\mathcal{T}\overset{\nabla_g}= \begin{pmatrix} \frac{2\sigma^{\frac{1}{2}}}{\sqrt{|\sigma^{-1}I^2}|}\omega \\ 0 \end{pmatrix}.$$
We can now state the following result:
\begin{prop}
In the notation of the preceding paragraph, let $\nabla \in \bm{p}$ and $F\in \Lambda^k\mathcal{T}^*$ be such that: $F_{A_1\dots A_k} \overset{\nabla}= k\mu_{a_2\dots a_k}Y_{[A_1}Z^{a_2}_{A_2}\cdots Z^{a_k}_{A_k]} + \xi_{a_1\dots a_k}Z^{a_1}_{A_1}\cdots Z^{a_k}_{A_k}$ then on the interior $M$: 
\begin{equation} \label{eq:tractor_hodge_star} \star F \overset{\nabla}= \begin{pmatrix} \frac{2\sigma^{\frac{1}{2}}}{\sqrt{|\sigma^{-1}I^2|}}\left((-1)^k\star \xi  + T\lrcorner(\star \mu)\right) \\ \frac{\sigma^{-\frac{3}{2}}I^2}{2\sqrt{|\sigma^{-1} I^2}|}\star \mu - \frac{2\sigma^{\frac{1}{2}}}{\sqrt{|\sigma^{-1}I^2}|}\left[ (-1)^kT^\flat \wedge (\star \xi) + T^\flat \wedge T\lrcorner(\star \mu) \right] \end{pmatrix}, \end{equation}
where: $T^b= -\frac{1}{n+1}\nabla_a\zeta^{ab}$, $\flat$ denotes the lowering of indices using $\zeta_{ab}$ and $\lrcorner$ denotes contraction.
\end{prop}

\begin{proof}
Let us first verify that the formula is reasonable on $M$ in the splitting determined by $\nabla_g$. Since $T$ is zero in this scale, the formula reduces to: 
$$ \star F \overset{\nabla_g}=\begin{pmatrix} (-1)^k\frac{2\sigma^{\frac{1}{2}}}{\sqrt{|\sigma^{-1}I^2|}} \star \xi \\ \frac{\sigma^{-\frac{3}{2}}I^2}{2\sqrt{|\sigma^{-1}I^2|}}\star \mu \end{pmatrix}.$$
Using Equation~\eqref{eq:tractor_wedge} let us calculate  $F\wedge \star F$ in the scale $\nabla_g$. The result is:
\begin{equation}\label{eq:hodgedualtrac1} F\wedge \star F = \begin{pmatrix}  \frac{\sigma^{-\frac{3}{2}}I^2}{2\sqrt{|\sigma^{-1} I^2}|}\mu \wedge \star \mu + \frac{2\sigma^{\frac{1}{2}}}{\sqrt{|\sigma^{-1}I^2|}}\xi \wedge \star \xi  \\ 0 \end{pmatrix}.   \end{equation}
Recall that the inner product $h$ on $\Lambda^k \mathcal{T}^*$ is defined by: $$h(F,G)=\frac{1}{k!}H^{A_1B_1}\dots H^{A_kB_k}F_{A_1\dots A_k}G_{B_1\dots B_k}.$$ Since, in the splitting given by $\nabla_g$, we have that: $H^{AB}\overset{\nabla_g}=\zeta^{ab}W^A_aW^B_b + \frac{1}{n}\zeta^{ab}P_{ab}X^AX^B$, hence: 
$$ h(F,F)= \frac{1}{n}\zeta^{ab}P_{ab}\zeta(\mu,\mu) + \zeta(\xi,\xi), $$
where, $P_{ab}$ is the projective Schouten tensor \emph{in the scale} $\nabla_g$ and, for any section of $\nu$ of $(\Lambda^kT^*M)(k+\omega)$, $\zeta(\nu,\nu)$ is shorthand for: \[\frac{1}{k!}\zeta^{a_1b_1}\cdots\zeta^{a_kb_k}\nu_{a_1\dots a_k}\nu_{b_1\dots b_k}.\]
In order to have a totally invariant formula, we observe that: \[I_A=D_A\sigma \overset{\nabla_g}= 2\sigma Y_A,\] thus:
\[I^2=\frac{4\sigma^2}{n}\zeta^{ab}P_{ab},\]
where $P_{ab}$ is calculated in the scale $\nabla_g$.
Hence: \[\frac{\zeta^{ab}P_{ab}}{n}=\frac{I^2\sigma^{-2}}{4}.\]
Consequently:
\[h(F,F)=\frac{I^2\sigma^{-2}}{4}\zeta(\mu,\mu)+\zeta(\xi,\xi).\]
Evaluating the top slot in Equation~\eqref{eq:hodgedualtrac1}, we find:
\[\begin{aligned} \frac{\sigma^{-\frac{3}{2}}I^2}{2\sqrt{|\sigma^{-1} I^2}|}\mu \wedge \star \mu + \frac{2\sigma^{\frac{1}{2}}}{\sqrt{|\sigma^{-1}I^2|}}\xi \wedge \star \xi &=  \frac{\sigma^{-\frac{3}{2}}I^2}{2\sqrt{|\sigma^{-1} I^2}|}\zeta(\mu,\mu)\omega\\ &\hspace{2.7cm}+ \frac{2\sigma^{\frac{1}{2}}}{\sqrt{|\sigma^{-1}I^2|}}\zeta(\xi,\xi)\omega,
\\&=h(F,F)\frac{2\sigma{\frac{1}{2}}}{\sqrt{|\sigma^{-1}I^2|}}\omega. \end{aligned}\]
Therefore:
\[F\wedge \star F =h(F,F)\Omega_\mathcal{T},\]
as desired. To verify that the result is correct for any connection in $\bm{p}$, we only need to verify that the components in Equation~\eqref{eq:tractor_hodge_star} transform correctly under a change of connection $\nabla \to \nabla + \Upsilon=\hat{\nabla}$. According to Equation~\eqref{eq:tractor_form_transform}, the top slot, $(TS)$, must be invariant. To check this, note that $\hat{\xi}=\xi + \Upsilon\wedge \mu$ and $\hat{\mu}=\mu.$
Furthermore, $\hat{T}^b = -\frac{1}{n+1} \hat{\nabla}_a \zeta^{ab}$ and:
$$\hat{\nabla}_c\zeta^{ab}= \nabla_c\zeta^{ab} + \Upsilon_d\zeta^{db}\delta_c^a + \Upsilon_d\zeta^{ad}\delta_c^b,$$
which leads to:
$$\hat{\nabla}_a \zeta^{ab} = \nabla_a \zeta^{ab} + (n+1)\Upsilon_a\zeta^{ab}.$$
Therefore: 
\begin{equation} \label{eq:Ttransform} \hat{T} = T - \Upsilon^\sharp \textrm{ and }\hat{T}^\flat = T^\flat - \Upsilon. \end{equation}
Since $\hat{\xi} = \xi + \Upsilon\wedge \mu$, Proposition~\ref{prop:hodgestarwedge} shows that: \[\star \hat{\xi} = \star \xi + \star(\Upsilon \wedge \mu) = \star \xi + (-1)^k\Upsilon^\sharp \lrcorner\star \mu.\]
Thus, overall: \[\begin{aligned} (-1)^k\star \hat{\xi} + \hat{T} \lrcorner \star \hat{\mu} &= (-1)^k\star \xi + \Upsilon^\sharp \lrcorner \star\mu -\Upsilon^\sharp\lrcorner\star \mu + T\lrcorner \star \mu,  \\ &= (-1)^k\star \xi  +T\lrcorner \star \mu, \end{aligned}\]
proving the projective invariance of the top slot as desired.

Referring again to Equation~\eqref{eq:tractor_form_transform}, we must now show that the bottom slot $(BS)$ of Equation~\eqref{eq:tractor_hodge_star} satisfies:
\[\hat{(BS)}=(BS)+ \Upsilon \wedge (TS). \]
Observe that, since, $\hat{\mu} = \mu$ the first term in $(BS)$ is invariant, moreover, the second term can be written: $- T^\flat \wedge(TS)$, so the desired result follows immediately from Equation~\eqref{eq:Ttransform}. 
\end{proof}

\subsection{Tractor co-differential and Hodge Laplacian}
In the previous sections we have sufficiently enhanced the structure on the tractor tensor algebra to introduce a tractor co-differential operator. Towards this, let $s$ denote the sign of the determinant of $\zeta^{ab}$ on $M$, and set: $$\varepsilon=\textrm{sgn}(\sigma^{-1}I^2);$$ observe that the sign of the determinant of $H^{AB}$ is then $s\varepsilon$. For future convenience, we introduce the notations:
\begin{equation}f=\sigma^{-1}I^2, \quad f'=\textrm{d}f.\end{equation}Now:
\begin{defi}
By analogy with the usual exterior calculus, we define a tractor codifferential on $k$-cotractors by: $$\mathscr{D}^* = (-1)^{(n+1)(k-1)+1}s\varepsilon\star \mathscr{D} \star.$$
\emph{In the scale} $\nabla_g$ on $M$, for any $F\overset{\nabla_g}{=} \begin{pmatrix}\mu\\\xi\end{pmatrix} \in \Gamma[(\Lambda^k\mathcal{T}^*)(\omega)]$:
\begin{equation} \label{eq:codiff}\mathscr{D}^*F \overset{\nabla_g}= \begin{pmatrix} \frac{1}{2f}(f')^\#\lrcorner \mu -\delta\mu 
\\  \frac{1}{2f}(f')^\# \lrcorner \xi + \delta \xi - \frac{(\omega+n+1-k)\sigma^{-1}f}{4}\mu 
\end{pmatrix}. \end{equation} 
In the above equation, we have introduced the notation: \[(f')^\# = \zeta^{ab}\nabla_b f \in \mathcal{E}(-2).\]
\end{defi}
\begin{proof}
We prove Equation~\eqref{eq:codiff}: applying successively Equations~\eqref{eq:tractor_hodge_star} and~\eqref{eq:tractor_diff_ext_def} to $F\overset{\nabla_g}{=} \begin{pmatrix}\mu\\ \xi \end{pmatrix} \in \Gamma[(\Lambda^k\mathcal{T}^*)(\omega)]$, we find:
\[\begin{aligned} \mathscr{D}\star F &= \begin{pmatrix} \frac{(\omega + n+1-k)\sigma^{-\frac{1}{2}}f}{2|f|^{\frac{1}{2}}}\star \mu + 2(-1)^{k+1} \dd\left(\left(\frac{\sigma}{|f|}\right)^{\frac{1}{2}}\star \xi \right) \\ \dd\left(\frac{\sigma^{-\frac{1}{2}}f}{2|f|^{\frac{1}{2}}}\star \mu\right) \end{pmatrix} ,
\\&=\begin{pmatrix}\frac{(\omega + n+1-k)\sigma^{-\frac{1}{2}}f}{2|f|^{\frac{1}{2}}}\star \mu + \frac{2\sigma^{\frac{1}{2}}(-1)^{k+1}}{|f|^{\frac{1}{2}}} \left( -\frac{1}{2f}f'\wedge \star \xi + \dd(\star \xi)\right)\\ \frac{\sigma^{-\frac{1}{2}}}{4|f|^{\frac{1}{2}}} f' \wedge \star \mu + \frac{\sigma^{-\frac{1}{2}}f}{2|f|^\frac{1}{2}}\dd\star\mu \end{pmatrix}.\end{aligned}\]
Applying once more the tractor Hodge star to this $n+1-(k-1)$ form of weight $\omega -1$, we find:
$$\begin{aligned} \star \mathscr{D}\star F &=\begin{pmatrix} \frac{(-1)^{n-k}}{2|f|}\star(f'\wedge\star\mu) +\frac{(-1)^{n-k}f}{|f|}\star d\star\mu 
\\ \frac{(\omega + n+ 1 -k)\sigma^{-1}f^2}{4|f|}\star\star\mu + \frac{(-1)^{k+1}f}{|f|}\left(-\frac{1}{2f} \star(f'\wedge \star \xi)+\star \dd \star \xi \right)\end{pmatrix} \\ &=(-1)^{(n+1)(k+1)}s\begin{pmatrix} \frac{-1}{2|f|} (f')^\#\lrcorner \mu +\frac{f}{|f|}\delta \mu \\ \frac{(\omega+n+1-k)\sigma^{-1}f^2}{4|f|}\mu - \frac{f}{|f|}\left( \frac{1}{2f}(f')^\#\lrcorner\xi +\delta\xi \right)   \end{pmatrix}. \end{aligned}$$
Equation~\eqref{eq:codiff} now follows from the fact that $\varepsilon = \textrm{sgn}(\sigma^{-1}I^2)=\textrm{sgn}(f)$.
\end{proof}

\subsection{The relationship between $\{\mathscr{D},\mathscr{D}^*\}$ and $\Delta^T$}
The exterior differential calculus we have developed above leads us to define a new D'Alembertian-type operator, analogous to the Hodge or de-Rham Laplacian. In general, it is to be distinguished from $H^{AB}D_AD_B$ that we studied in Section~\eqref{sec:projective_laplace1}. The remainder of this section is devoted to obtaining an expression for $\mathscr{D}^*\mathscr{D} + \mathscr{D}\mathscr{D}^*=\{\mathscr{D},\mathscr{D}^*\}$. In order to simplify the computation, we will work exclusively with the Levi-Civita connection $\nabla_g $ of $g=\sigma^{-1}\zeta^{-1}$ on $M$; any identities will extend by density to $\overline{M}$. We first remark that the weight $2$ density $I^2$ is, in general, not parallel in the scale $\nabla_g$.
\begin{lemm}
In an arbitrary scale $\nabla \in \bm{p}$:
\[\nabla_c I^2 = \frac{8\sigma^2}{n}\zeta^{ef}Y_{cef} -\frac{4\sigma}{n}\nabla_a\sigma \zeta^{ef}W_{ce\phantom{b}f}^{\phantom{ce}b}.\]
\end{lemm}
\begin{proof}
\[\nabla_c{I^2}=(\nabla_cH^{AB})I_AI_B + 2H^{AB}(\nabla_aI_A)I_B.\]
The second term is easily seen to cancel, and we evaluate the first one using Equation~\eqref{eq:me2} and the calculations we did in Section~\ref{sec:me}.
Finally, we have $I_A=2\sigma Y^A + \nabla_a\sigma Z^a_A$, thus:
\[2X^{(A}W^{\phantom{cE}B)}_{cE\phantom{B}F}H^{EF}I_AI_B=-8\sigma^2\zeta^{ef}Y_{cef} +4\sigma\nabla_a\sigma \zeta^{ef}W_{ce\phantom{b}f}^{\phantom{ce}b}.\]
\end{proof}
 Let $F$ denote an arbitrary section of $(\Lambda^k \mathcal{T}^*)(\omega)$ given in the scale $\nabla_g$ by:
\[F \overset{\nabla_g}=\begin{pmatrix}\mu\\\xi\end{pmatrix}.\]
Equation~\eqref{eq:tractor_diff_ext_def} yields directly the expression for $\mathscr{D}F$, which, due to the symmetry of $P_{ab}$, simplifies to:
\[\mathscr{D}F \overset{\nabla_g}=\begin{pmatrix} (\omega+k)\xi- \dd\mu \\ \dd\xi \end{pmatrix}.\]
In the above expression $\dd$ denotes, abusively, the \emph{covariant} exterior derivative on the weighted bundles. Now:
\[\star \mathscr{D}F = \begin{pmatrix} (-1)^{k+1}\frac{2\sigma^{\frac{1}{2}}}{|f|^{\frac{1}{2}}}\star \dd\xi  \\ \sigma^{-\frac{1}{2}} \frac{f}{2|f|^{\frac{1}{2}}} \left( (\omega +k)\star\xi -\star \dd\mu \right) \end{pmatrix}.\]
We apply again $\mathscr{D}$, observing first that $\star\mathscr{D}F$ is a $(n+1-(k+1))=n-k$ form of weight $\omega -1$, thus:
\[\mathscr{D}\star\mathscr{D} F\overset{\nabla_g}=\begin{pmatrix}{\frac{(\omega -1+n-k)f}{2(|f|\sigma)^{\frac{1}{2}}}}\left((\omega+k)\star\xi -\star d\mu\right) +2(-1)^{k}\dd\left(\frac{\sigma^{\frac{1}{2}}}{|f|^{\frac{1}{2}}}\star \dd\xi \right) \\  \dd\left( \sigma^{-\frac{1}{2}} \frac{f}{2|f|^{\frac{1}{2}}} \left( (\omega +k)\star\xi -\star \dd\mu \right)\right)\end{pmatrix}.\]
For readability, we will now proceed to treat the top and bottom slots separately. Let us begin with the unevaluated differential in the top slot: 
\[\dd\left(\frac{\sigma^{\frac{1}{2}}}{|f|^{\frac{1}{2}}}\star \dd\xi \right)=-\sigma^{\frac{1}{2}}\frac{f'\wedge \star \dd\xi}{2f|f|^{\frac{1}{2}}} +\frac{\sigma^\frac{1}{2}}{|f|^\frac{1}{2}}\dd\star \dd\xi.\]
Downstairs we have:
\[\frac{\sigma^{-\frac{1}{2}}}{4|f|^{\frac{1}{2}}}f'\wedge((\omega+k)\star\xi-\star \dd\mu) + \frac{\sigma^{-\frac{1}{2}}f}{2|f|^{\frac{1}{2}}}\left( (\omega +k)\dd\star\xi -\dd\star \dd\mu \right).\]
Applying again the Hodge star to this $(n+1-k)$-tractor form of weight $(\omega -2)$ leads to a new $k$-tractor form of weight $\omega -2$ with, in the top slot:
\[\begin{split}\frac{(-1)^{1+k(n-k)}s}{2|f|}\left((\omega+k)(f')^\#\lrcorner \xi - (f')^\#\lrcorner \dd\mu \right) +\frac{(-1)^{k(n-k)}sf}{|f|}\big((\omega+k)\delta\xi \\-\delta \dd\mu\big).\end{split}\]
 As for the bottom slot, it evaluates to:
\[\begin{split}\frac{(\omega-1+n-k)f^2(-1)^{k(n-k)}s}{4|f|\sigma}((\omega+k)\xi -\dd\mu)\hspace{1in}\\\hspace{1in} +\frac{(-1)^{k(n-k)+1}s}{|f|}\left(\frac{(f')^\#\lrcorner \dd\xi}{2} +f\delta \dd\xi \right).\end{split}\]
Overall, in the scale $\nabla_g$, $\star \mathscr{D}\star \mathscr{D}F$ evaluates to:
\[\begin{pmatrix}\frac{(-1)^{1\!+\!k(n\!-\!k)}s}{2|f|}\left((\omega\!+\!k)(f')^\#\lrcorner \xi \!-\! (f')^\#\lrcorner d\mu \right)\!+\!\frac{(-1)^{k(n\!-\!k)}sf}{|f|}((\omega\!+\!k)\delta\xi -\delta d\mu) \\\\ \frac{(\omega\!-\!1\!+\!n\!-\!k)f^2(-1)^{k(n\!-\!k)}s}{4|f|\sigma}((\omega\!+\!k)\xi \!-\!d\mu)\!+\!\frac{(-1)^{k(n\!-\!k)\!+\!1}s}{|f|}\left( \frac{(f')^\#\lrcorner d\xi}{2}\!+\!f\delta d\xi \right) \end{pmatrix}\!\!.\]
It remains only to correct the sign ! The result is:
\[\mathscr{D}^*\mathscr{D} F \overset{\nabla_g}=\begin{pmatrix}\frac{1}{2f}\left((\omega+k)(f')^\#\lrcorner \xi - (f')^\#\lrcorner d\mu \right) -((\omega+k)\delta\xi -\delta d\mu) \\\\ -\frac{(\omega-1+n-k)f}{4\sigma}((\omega+k)\xi -d\mu) +\left( \frac{(f')^\#\lrcorner d\xi}{2f} +\delta d\xi \right) \end{pmatrix}.\]
Calculating $\mathscr{D}\mathscr{D}^*$ is slightly less involved as we have only to apply Equation~\eqref{eq:tractor_diff_ext_def} to Equation~\eqref{eq:codiff}, taking care to note that $\mathscr{D}^*F$ is a $k-1$ form of weight $\omega -1$. We find that the expresion for $\mathscr{D}\mathscr{D}^*F$ in the scale $\nabla_g$ is given by:
\[ \begin{pmatrix} \begin{split}- (\omega\!+\!k\!-\!2)(\omega\!+\!n\!+\!1\!-\!k)\frac{\sigma^{-1}f}{4}\mu + (\omega\!+\!k\!-\!2)\left( \frac{1}{2f}(f')^\#\lrcorner\xi +\delta\xi \right) \\-d(\frac{1}{2f}(f')^\#\lrcorner \mu) +d\delta \mu\end{split} \\ \\\hline \\ d\delta\xi +d(\frac{1}{2f}(f')^\#\lrcorner \xi) - (\omega\!+\!n\!+\!1\!-\!k)\frac{\sigma^{-1}}{4}d(f\mu)  \end{pmatrix}.\]
We summarise these computations in:
\begin{prop}
In the scale $\nabla_g$, $\{\mathscr{D}^*,\mathscr{D}\}$ acts on a section $F$ of $(\Lambda^k\mathcal{T}^*)(\omega)$ expressed as:
\[F\overset{\nabla_g}=\begin{pmatrix} \mu \\ \xi \end{pmatrix}, \]  according to: 
\begin{equation}
\label{eq:anti_comm_exterior}
\begin{pmatrix}
\begin{split}\{\dd,\delta\}\mu - 2\delta\xi +(\omega\!+\!k\!-\!1)\frac{1}{f}\vec{\nabla}f \lrcorner \mu - \frac{1}{2} \mathcal{L}_{f^{-1}\vec{\nabla} f} \mu \\-(\omega\!+\!k\!-\!2)(\omega\!+\!n\!+\!1\!-\!k)\frac{\sigma^{-1}f}{4}\mu\end{split} \\\\\hline\\
\begin{split}\{\dd,\delta\}\xi\!-\!\frac{f\sigma^{-1}}{2}\dd\mu\!-\!(\omega\!+\!n\!+\!1\!-\!k)\frac{\sigma^{-1}}{4}(f'\wedge\mu)\!+\!\frac{1}{2}\mathcal{L}_{f^{-1}\vec{\nabla} f}\xi\!\\-(\omega\!+\!k)(\omega\!-\!1\!+\!n\!-\!k)\frac{f\sigma^{-1}}{4}\xi\end{split}
\end{pmatrix},
\end{equation}

In the above we have introduced the Lie derivative $\mathcal{L}_X$ extended to weighted vector fields $X$ by the formula:
\[\mathcal{L}_X \xi_{a_1\dots a_k}=X^a \nabla_a\xi_{a_1\dots a_k} + k(\nabla_{[a_1}X^a)\xi_{|a|a_2\dots a_k]}.\]
\end{prop}

\subsection{Weitzenbock identity}
Having already introduced the D'Alembertian type operator $\Delta^\mathcal{T}=H^{AB}D_AD_B$ on generic tractor $k$-coforms in Section~\ref{sec:projective_laplace1}, it is interesting to explore how it compares to $\{ \mathscr{D}, \mathscr{D}^*\}$. It turns out the relationship between them is completely analogous to that between the Bochner Laplacian and the de-Rham Laplacian on the base manifold. As before, we perform all calculations in the Levi-Civita scale $\nabla_g$.   Recall from Equation~\eqref{eq:laplac} that then:
\begin{equation*} H^{AB}D_{A}D_{B}F\overset{\nabla_g }= \frac{(\omega+n-1)\omega f\sigma^{-1}}{4}F  +\zeta^{ab}\nabla_a\nabla_b F, \end{equation*}
where for a weighted $k$-cotractor form in an arbitrary scale $\nabla$: 
\[\nabla_a \nabla_b F \overset{\nabla}= \begin{pmatrix} \nabla_a\nabla_b\mu_{a_2\dots a_k} - 2\nabla_{(a}\xi_{b)a_2\dots a_k} -kP_{b[a}\mu_{a_2\dots a_k]} \\\\
\begin{split}\nabla_a\nabla_b\xi_{a_1\dots a_k} +2kP_{(a|[a_1}\nabla_{|b)|}\mu_{a_2\dots a_k]} +k(\nabla_a P_{b[a_1})\mu_{a_2\dots a_k]}\\ -kP_{ a[a_1}\xi_{|b|a_2\dots a_k]}\end{split}\end{pmatrix}. \]

To simplify computations a little, we restrict now our attention to \emph{normal} solutions of the Metrisability equation, for which a number of terms in the above expressions, and in particular Equation~\eqref{eq:anti_comm_exterior}, vanish.
Furthermore, in the scale $\nabla_g$, it is easily seen that:
\begin{equation}\label{eq:normal_sol_schouten2} P_{cd}=\frac{f\sigma^{-1}}{4}\zeta_{cd}. \end{equation}
Finally, $\nabla P_{cd}$ and all derivatives of $f$ vanish. Overall, performing all the preceding simplifications, we have, for an arbitrary $k$ form: \[H^{AB}D_A D_B F \overset{\nabla_g}=\begin{pmatrix} \Box \mu +2\delta \xi + (\omega(\omega+n-1) -(n+1-k))\frac{f\sigma^{-1}}{4}\mu \\\\ \Box \xi +\frac{f\sigma^{-1}}{2}\textrm{d}\mu+(\omega(\omega+n-1)-k)\frac{f\sigma^{-1}}{4}\xi \end{pmatrix},\]
where we define: $\Box \mu=\zeta^{ab}\nabla_a \nabla_b \mu$.

Similarly, Equation~\eqref{eq:anti_comm_exterior} simplifies to:
\begin{equation*}
\{\mathscr{D},\mathscr{D}^*\}F\overset{\nabla_g}{=}\begin{pmatrix}
\{\dd,\delta\}\mu - 2\delta\xi  -(\omega\!+\!k\!-\!2)(\omega\!+\!n\!+\!1\!-\!k)\frac{\sigma^{-1}f}{4}\mu \\ \\
\{\dd,\delta\}\xi\!-\!\frac{f\sigma^{-1}}{2}\dd\mu\!-\!(\omega\!+\!k)(\omega\!-\!1\!+\!n\!-\!k)\frac{f\sigma^{-1}}{4}\xi
\end{pmatrix}.
\end{equation*}
From these expressions we will show:
\begin{prop}
\label{prop:weitzenbock}
Let $F\in \Gamma((\Lambda^k\mathcal{T}^*)(\omega))$ and suppose that $H^{AB}$ is a normal solution to the Metrisability equation, then:
\[\begin{aligned} (\{\mathcal{D},\mathcal{D}^*\}F)_{A_1\dots A_k} = &- (H^{AB}D_AD_BF)_{A_1\dots A_k} \\&+k(k+1)H^{AB}\Omega_{[A|B|\phantom{C}A_1}^{\phantom{A|B|}C}F^{\phantom{C}}_{|C|A_2\dots A_k]}.\end{aligned}\]
In the above, $\Omega_{AB\phantom{C}D}^{\phantom{AB}C}=\Omega_{ab\phantom{C}D}^{\phantom{ab}C}Z_A^aZ_B^a$ and $\Omega_{ab\phantom{C}D}^{\phantom{ab}C}$ is the tractor curvature tensor. (cf. Equation~\eqref{eq:courbure_trac}).
\end{prop} 
\begin{proof}
We first inspect the difference between the order zero terms in each slot of the two tractors:
\begin{equation}\label{eq:order_zero_weit} \begin{aligned} (\omega\!+\!k\!-\!2)(\omega\!+\!n\!+\!1\!-\!k)&\!=\! \omega(\omega\!+\!n\!-\!1)\!-\!(n\!+\!1\!-\!k)\!-\!(k\!-\!1)(n\!+\!1\!-\!k),\\ (\omega+k)(\omega+n-1-k)&=\omega(\omega+n-1)-k+k(n-k). \end{aligned} \end{equation}
Moreover, in index notation the usual Weitzenbock identity extended to weighted tensors reads:
\begin{equation} \label{eq:weitzenbock_usuelle} \begin{split}{\{\dd,\delta\}\xi}_{\substack{\\a_1\dots a_k}} \!\!\!+ \zeta^{ab}\nabla_a\nabla_b\xi_{\substack{\\a_1\dots a_k}} = \sum_{i=1}^k \zeta^{ab}R^{\phantom{a_i a}c}_{a_i a\phantom{c}b}\xi_{a_1\dots a_{i-1}c a_{i+1}\dots a_k} \\+\sum_{i=1}^k\sum_{\overset{j=1}{j\neq i}}^k\zeta^{ab}R^{\phantom{a_ia}c}_{a_ia\phantom{c}a_j}\xi_{a_1\dots a_{j-1}ca_{j+1}\dots a_{i-1}ba_{i+1}\dots a_k}.\end{split}\end{equation}
Now, appealing to Equations~\eqref{eq:decomposition_riemann} and~\eqref{eq:normal_sol_schouten2}, we have:
$$R_{ab\phantom{c}d}^{\phantom{ab}c} =W_{ab\phantom{c}d}^{\phantom{ab}c}+ \left(\delta_a^{c}\zeta_{bd}-\delta_b^c\zeta_{ad} \right)\frac{f\sigma^{-1}}{4}.$$
Therefore:
\[\begin{aligned} \zeta^{ab}R^{\phantom{a_i a}c}_{a_i a\phantom{c}b}\xi_{a_1\dots a_{i-1}c a_{i+1}\dots a_k}&=\zeta^{ab}W^{\phantom{a_i a}c}_{a_i a\phantom{c}b}\xi_{a_1\dots a_{i-1}c a_{i+1}\dots a_k}\\&\hspace{.4in}+\frac{f\sigma^{-1}}{4}\underbrace{\zeta^{ab}\left(\delta_{a_i}^{c}\zeta_{ab}-\delta_a^c\zeta_{a_ib} \right)}_{(n-1)\delta_{a_i}^c}\xi_{a_1\dots a_{i-1}c a_{i+1}\dots a_k},\\ &= \zeta^{ab}W^{\phantom{a_i a}c}_{a_i a\phantom{c}b}\xi_{a_1\dots a_{i-1}c a_{i+1}\dots a_k} +\frac{f\sigma^{-1}}{4}(n-1)\xi_{a_1\dots a_k}, \end{aligned}\]
and:
\[\begin{split} \zeta^{ab}R^{\phantom{a_ia}c}_{a_ia\phantom{c}a_j}\xi_{a_1\dots a_{j-1}ca_{j+1}\dots a_{i-1}ba_{i+1}\dots a_k}= \zeta^{ab}W^{\phantom{a_ia}c}_{a_ia\phantom{c}a_j}\xi_{a_1\dots a_{j-1}ca_{j+1}\dots a_{i-1}ba_{i+1}\dots a_k} \\ + \frac{f\sigma^{-1}}{4} \underbrace{\zeta^{ab}\left(\delta_{a_i}^{c}\zeta_{aa_j}-\delta_{a}^c\zeta_{a_ia_j} \right)\xi_{a_1\dots a_{j-1}ca_{j+1}\dots a_{i-1}ba_{i+1}\dots a_k}}_{= - \xi_{a_1\dots a_k} } \end{split}.\]
So Equation~\eqref{eq:weitzenbock_usuelle} can be written:
\[\begin{aligned} {\{d,\delta\}\xi}_{a_1\dots a_k} + \zeta^{ab}\nabla_a\nabla_b\xi_{a_1\dots a_k} = & \left[\begin{array}{c}\textrm{Terms involving}\\ \textrm{Weyl tensor}\end{array}\right] \\&+\underbrace{[(n-1)k-k(k-1)]}_{=k(n-k)}\frac{f\sigma^{-1}}{4} \xi_{a_1\dots a_k}.\end{aligned}\]
The second term in the above equation accounts exactly for the differences observed in Equation~\eqref{eq:order_zero_weit} and it follows that in the scale $\nabla_g$, $\{\mathcal{D},\mathcal{D}^*\}F + H^{AB}D_AD_BF$ is given by:
\[ \begin{pmatrix}\begin{split} \sum_{i=1}^{k-1} \zeta^{ab}&W^{\phantom{a_i a}c}_{a_i a\phantom{c}b}\mu_{a_1\dots a_{i-1}c a_{i+1}\dots a_{k-1}} \\&+\sum_{i=1}^{k-1}\sum_{\overset{j=1}{j\neq i}}^{k-1}\zeta^{ab}W^{\phantom{a_ia}c}_{a_ia\phantom{c}a_j}\mu_{a_1\dots a_{j-1}ca_{j+1}\dots a_{i-1}ba_{i+1}\dots a_{k-1}}.\end{split} \\\\\hline\\ \begin{split}\sum_{i=1}^k \zeta^{ab}&W^{\phantom{a_i a}c}_{a_i a\phantom{c}b}\xi_{a_1\dots a_{i-1}c a_{i+1}\dots a_k} \\&+\sum_{i=1}^k\sum_{\overset{j=1}{j\neq i}}^k\zeta^{ab}W^{\phantom{a_ia}c}_{a_ia\phantom{c}a_j}\xi_{a_1\dots a_{j-1}ca_{j+1}\dots a_{i-1}ba_{i+1}\dots a_k}.\end{split} \end{pmatrix} \]
We must now attempt to identify the tractor on the right-hand side of the above equation. We claim that it is exactly: \[k(k+1)H^{AB}\Omega_{[A|B|\phantom{C}A_1}^{\phantom{A|B|}C}F^{\phantom{C}}_{|C|A_2\dots A_k]},\] where: \[\Omega^{\phantom{AB}C}_{AB\phantom{C}D}=\Omega^{\phantom{ab}C}_{ab\phantom{C}D}Z_A^aZ_B^b \overset{\nabla_g}{=} W^{\phantom{ab}c}_{ab\phantom{c}d}W^C_cZ_A^aZ_B^bZ^d_D.\]
The calculation is \enquote{merely} technical and presents no conceptual subtleties, therefore we will only carry it out here fully on the bottom component and leave the upper component to the reader.
We first write:
\[k(k+1)H^{AB}\Omega_{[A|B|\phantom{C}A_1}^{\phantom{A|B|}C}F^{\phantom{C}}_{|C|A_2\dots A_k]}\!=\!\frac{H^{AB}\displaystyle \sum_{\sigma \in \mathfrak{S}_{k+1}}\hspace{-3.5mm}\varepsilon(\sigma)\Omega^{\phantom{D_{\sigma\!(\!1\!)}B}C}_{D_{\sigma\!(\!1\!)}B\phantom{C}D_{\sigma(\!2\!)}}F_{CD_{\sigma\!(\!3\!)}\dots D_{\sigma\!(\!k\!+\!1\!)}}}{{(k-1)!}}.\]
For convenience, we have introduced a new set of indices $\{D_i\}$ defined by: \[D_1=A, D_{i}=A_{i-1}, i\geq 2.\]
Specialising to the Levi-Civita scale $\nabla_g$, the right-hand side is:
\[\frac{H^{AB}}{(k-1)!}\sum_{\sigma \in \mathfrak{S}_{k+1}}\varepsilon(\sigma)W^{\phantom{d_1b}c}_{d_1b\phantom{c}d_{2}}W^C_cZ_{D_{\sigma(1)}}^{d_1}Z_B^bZ^{d_2}_{D_{\sigma(2)}}F_{CD_{\sigma(3)}\dots D_{\sigma(k+1)}}.\]
Now: $H^{AB}Z_B^b \overset{\nabla_g}= \zeta^{ab}W_{A}^a$ so we must calculate:
$$T=\sum_{\sigma \in \mathfrak{S}_{k+1}}\varepsilon(\sigma)\zeta^{ab}W_a^AW^{\phantom{d_1b}c}_{d_1b\phantom{c}d_{2}}W^C_cZ_{D_{\sigma(1)}}^{d_1}Z^{d_2}_{D_{\sigma(2)}}F_{CD_{\sigma(3)}\dots D_{\sigma(k+1)}}.$$
In order to isolate the bottom slot, consider $F_{A_1\dots A_k} \overset{\nabla_g}{=} \xi_{a_1\dots a_k}Z^{a_1}_{A_1}\dots Z^{a_k}_{A_k}$.
In this case, 
\[\begin{aligned} T&=\!\!\!\!\sum_{\sigma \in \mathfrak{S}_{k+1}}\!\!\!\varepsilon(\sigma)\zeta^{ab}W_a^AW^{\phantom{d_1b}c}_{d_1b\phantom{c}d_{2}}\xi_{\tilde{c}d_3\dots d_{k+1}}\underbrace{W^C_cZ^{\tilde{c}}_C}_{=\delta_c^{\tilde{c}}}Z_{D_{\sigma(1)}}^{d_1}Z^{d_2}_{D_{\sigma(2)}}Z^{d_3}_{D_{\sigma(3)}}\dots Z^{d_{k+1}}_{D_{\sigma(k+1)}},  \\&=\sum_{\sigma \in \mathfrak{S}_{k+1}}\!\!\!\varepsilon(\sigma)\zeta^{ab}W_a^AW^{\phantom{d_1b}c}_{d_1b\phantom{c}d_{2}}\xi_{cd_3\dots d_{k+1}}Z_{D_{\sigma(1)}}^{d_1}Z^{d_2}_{D_{\sigma(2)}}Z^{d_3}_{D_{\sigma(3)}}\dots Z^{d_{k+1}}_{D_{\sigma(k+1)}}.\end{aligned}\]
Observe now that $T$ can be rewritten:
\[\begin{aligned} T&=\sum_{\sigma \in \mathfrak{S}_{k+1}}\varepsilon(\sigma)\zeta^{ab}W_a^AW^{\phantom{d_{\sigma(1)}b}c}_{d_{\sigma(1)}b\phantom{c}d_{\sigma(2)}}\xi_{cd_{\sigma(3)}\dots d_{\sigma(k+1)}}Z_{D_1}^{d_1}Z^{d_2}_{D_2}Z^{d_3}_{D_3}\dots Z^{d_{k+1}}_{D_{k+1}},\\&=\sum_{\sigma \in \mathfrak{S}_{k+1}}\varepsilon(\sigma)\zeta^{ab}W_a^AW^{\phantom{d_{\sigma(1)}b}c}_{d_{\sigma(1)}b\phantom{c}d_{\sigma(2)}}\xi_{cd_{\sigma(3)}\dots d_{\sigma(k+1)}}Z_{A}^{d_1}Z^{d_2}_{A_1}Z^{d_3}_{A_2}\dots Z^{d_{k+1}}_{A_{k}},\\&=\sum_{\sigma \in \mathfrak{S}_{k+1}}\varepsilon(\sigma)\zeta^{d_1b}W^{\phantom{d_{\sigma(1)}b}c}_{d_{\sigma(1)}b\phantom{c}d_{\sigma(2)}}\xi_{cd_{\sigma(3)}\dots d_{\sigma(k+1)}}Z^{d_2}_{A_1}Z^{d_3}_{A_2}\dots Z^{d_{k+1}}_{A_{k}}. \end{aligned}\]
If $\sigma(1)=1$, then the summand vanishes leading to: 
$$T = \sum_{\underset{\sigma(1)\neq 1}{\sigma \in {\mathfrak{S}_{k+1}}}}\varepsilon(\sigma)\zeta^{d_1b}W^{\phantom{d_{\sigma(1)}b}c}_{d_{\sigma(1)}b\phantom{c}d_{\sigma(2)}}\xi_{cd_{\sigma(3)}\dots d_{\sigma(k+1)}}Z^{d_2}_{A_1}Z^{d_3}_{A_2}\dots Z^{d_{k+1}}_{A_{k}}. $$
We now seek to exploit the antisymmetry of $\xi$, first we note that:
$$T = \sum_{i=2}^{k+1}\sum_{\underset{\sigma(1)=i}{\sigma \in {\mathfrak{S}_{k+1}}}}\varepsilon(\sigma)\zeta^{d_1b}W^{\phantom{d_ib}c}_{d_{i}b\phantom{c}d_{\sigma(2)}}\xi_{cd_{\sigma(3)}\dots d_{\sigma(k+1)}}Z^{d_2}_{A_1}Z^{d_3}_{A_2}\dots Z^{d_{k+1}}_{A_{k}}.$$
It is interesting to split the inner sum into two further sums as follows:
\[\begin{aligned}&\overbrace{\sum_{\underset{\underset{\sigma(2)=1}{\sigma(1)=i}}{\sigma \in {\mathfrak{S}_{k+1}}}}\varepsilon(\sigma)\zeta^{d_1b}W^{\phantom{d_ib}c}_{d_{i}b\phantom{c}d_{1}}\xi_{cd_{\sigma(3)}\dots d_{\sigma(k+1)}}Z^{d_2}_{A_1}Z^{d_3}_{A_2}\dots Z^{d_{k+1}}_{A_{k}}}^{=T_1^i} \\& \hspace{.2in}+\underbrace{\sum_{\underset{j\neq i}{j=2}}^{k+1}\sum_{\underset{\underset{\sigma(1)=i}{\sigma(2)=j}}{\sigma \in {\mathfrak{S}_{k+1}}}}\varepsilon(\sigma)\zeta^{d_1b}W^{\phantom{d_ib}c}_{d_{i}b\phantom{c}d_{j}}\xi_{cd_{\sigma(3)}\dots d_{\sigma(k+1)}}Z^{d_2}_{A_1}Z^{d_3}_{A_2}\dots Z^{d_{k+1}}_{A_{k}}}_{=T_2^i}.\end{aligned}\]
Consider now a generic term in the sum $T_1^i$. If we define: $$\tilde{d}_1=c, \tilde{d}_{l}=d_{\sigma(l+1)}, 2\leq l \leq k;$$ then:
$$\xi_{\tilde{d}_1\dots \tilde{d}_k}=\varepsilon(s)\xi_{\tilde{d}_{s(1)}\dots \tilde{d}_{s(k)}}, s \in \mathfrak{S}_k.$$
Considering the permutation $s$ given by:
$$s(l)=\begin{cases} \sigma^{-1}(l)-1 & \textrm{if $1\leq l < i$},\\ \sigma^{-1}(l+1)-1 &\textrm{ if $i \leq l \leq k$}, \end{cases}$$
which satisfies $s(1)=1$, and, by Appendix~\ref{app:preuve_lemme_tractorformcouple}, $\varepsilon(s)=(-1)^{i-1}\varepsilon(\sigma)$, then we have:
\[\begin{aligned}\xi_{cd_{\sigma(3)}\dots d_{\sigma(k+1)}}=\xi_{\tilde{d}_1\dots \tilde{d}_k}&=(-1)^{i-1}\varepsilon(\sigma)\xi_{\tilde{d}_{s(1)}\dots \tilde{d}_{s(k)}}\\&=(-1)^{i-1}\varepsilon(\sigma)\xi_{c d_2\dots d_{i-1}d_{i+1}d_{k+1}}.\end{aligned}\]
Overall: 
$$T_1^i = (k-1)!(-1)^{i-1} \zeta^{d_1b}W^{\phantom{d_ib}c}_{d_{i}b\phantom{c}d_{1}}\xi_{cd_{2}\dots d_{i-1}d_{i+1} d_{k+1}}Z^{d_2}_{A_1}Z^{d_3}_{A_2}\dots Z^{d_{k+1}}_{A_{k}}.$$
Hence:
$$\begin{aligned} \frac{1}{(k-1)!}\sum_{i=2}^{k+1}T_k^i&=\sum_{i=2}^{k+1}(-1)^{i-1} \zeta^{d_1b}W^{\phantom{d_ib}c}_{d_{i}b\phantom{c}d_{1}}\xi_{cd_{2}\dots d_{i-1}d_{i+1} d_{k+1}}Z^{d_2}_{A_1}Z^{d_3}_{A_2}\dots Z^{d_{k+1}}_{A_{k}}, \\&=\sum_{i=2}^{k+1} (-1)^{i-1}\zeta^{ab}W^{\phantom{d_{i-1}b}c}_{d_{i-1}b\phantom{c}a}\xi_{cd_{1}\dots d_{i-2}d_{i} d_{k}}Z^{d_1}_{A_1}Z^{d_2}_{A_2}\dots Z^{d_{k}}_{A_{k}},\\&=\sum_{i=1}^{k}\zeta^{ab}W^{\phantom{d_{i}b}c}_{d_{i}b\phantom{c}a}\xi_{d_{1}\dots d_{i-1}cd_{i+1} d_{k}}Z^{d_1}_{A_1}Z^{d_2}_{A_2}\dots Z^{d_{k}}_{A_{k}}. \end{aligned}  $$
We move on now to study a generic term $a(\sigma,i,j)$ in $T_2^i$:
$$ a(\sigma,i,j)=\varepsilon(\sigma)\zeta^{d_1b}W^{\phantom{d_ib}c}_{d_{i}b\phantom{c}d_{j}}\xi_{cd_{\sigma(3)}\dots d_{\sigma(k+1)}}Z^{d_2}_{A_1}Z^{d_3}_{A_2}\dots Z^{d_{k+1}}_{A_{k}},$$
It can be handled by the same reasoning as before, but now we should distinguish between the cases $i<j$ and $j>i$. In the first case, $s(j)=1$ and:
$$a(\sigma,i,j)= (-1)^{i-1}\zeta^{d_1b}W^{\phantom{d_ib}c}_{d_{i}b\phantom{c}d_{j}}\xi_{d_1\dots d_{j-1}cd_{j+1}\dots d_{i-1}d_{i+1}\dots d_{k+1}}Z^{d_2}_{A_1}Z^{d_3}_{A_2}\dots Z^{d_{k+1}}_{A_{k}},$$
In the second case: $s(j-1)=1$, and:
$$a(\sigma,i,j)= (-1)^{i-1}\zeta^{d_1b}W^{\phantom{d_ib}c}_{d_{i}b\phantom{c}d_{j}}\xi_{d_1\dots d_{i-1}d_{i+1}\dots d_{j-1}cd_{j+1}\dots d_{k+1}}Z^{d_2}_{A_1}Z^{d_3}_{A_2}\dots Z^{d_{k+1}}_{A_{k}}.$$
Overall: 
\[ \begin{aligned}\frac{T_2^i}{(k-1)!} &= \sum_{j=2}^{i-1}(-1)^{i-1}\zeta^{d_1b}W^{\phantom{d_ib}c}_{d_{i}b\phantom{c}d_{j}}\xi_{d_1\dots d_{j\!-\!1}cd_{j\!+\!1}\dots d_{i\!-\!1}d_{i\!+\!1}\dots d_{k\!+\!1}}Z^{d_2}_{A_1}Z^{d_3}_{A_2}\dots Z^{d_{k\!+\!1}}_{A_{k}} \\&+\sum_{j=i+1}^{k+1}\!(-1)^{i-1}\zeta^{d_1b}W^{\phantom{d_ib}c}_{d_{i}b\phantom{c}d_{j}}\xi_{d_1\dots d_{i\!-\!1}d_{i\!+\!1}\dots d_{j\!-\!1}cd_{j\!+\!1}\dots d_{k\!+\!1}}Z^{d_2}_{A_1}Z^{d_3}_{A_2}\dots Z^{d_{k\!+\!1}}_{A_{k}}. \end{aligned}\]
Reindex now as follows:
\[\begin{split}\sum_{j=2}^{i-1}(-1)^{i-1}\zeta^{ab}W^{\phantom{d_ib}c}_{d_{i}b\phantom{c}d_{j-1}}\xi_{ad_1\dots d_{j-2}cd_{j}\dots d_{i-2}d_{i}\dots d_{k}}Z^{d_1}_{A_1}Z^{d_2}_{A_2}\dots Z^{d_{k}}_{A_{k}},\hspace{1cm} \\= \sum_{j=1}^{i-2}(-1)^{i-1}\zeta^{ab}W^{\phantom{d_ib}c}_{d_{i}b\phantom{c}d_{j}}\xi_{ad_1\dots d_{j-1}cd_{j+1}\dots d_{i-2}d_{i}\dots d_{k}}Z^{d_1}_{A_1}Z^{d_2}_{A_2}\dots Z^{d_{k}}_{A_{k}},\end{split}\]
and similarly for the second sum so that:
\[\begin{aligned} \frac{1}{(k-1)!}\sum_{i=2}^{k+1}T^i_2&=\frac{1}{(k-1)!}\sum_{i=1}^{k}T_2^{i+1}, \\&\hspace{-1.1cm}= \sum_{i=1}^k\sum_{\underset{j\neq i}{j=1}}^{k}\zeta^{ab}W^{\phantom{d_ib}c}_{d_{i}b\phantom{c}d_{j}}(-1)^{i}\xi_{ad_1\dots d_{j\!-\!1}cd_{j\!+\!1}\dots d_{i\!-\!1}d_{i\!+\!1}\dots d_{k}}Z^{d_1}_{A_1}Z^{d_2}_{A_2}\dots Z^{d_{k}}_{A_{k}},\\&\hspace{-1.1cm}= \sum_{i=1}^k\sum_{\underset{j\neq i}{j=1}}^{k}\zeta^{ab}W^{\phantom{d_ib}c}_{d_{i}b\phantom{c}d_{j}}\xi_{d_1\dots d_{j-1}cd_{j+1}\dots d_{i\!-\!1}ad_{i\!+\!1}\dots d_{k}}Z^{d_1}_{A_1}Z^{d_2}_{A_2}\dots Z^{d_{k}}_{A_{k}}.\end{aligned} \]
This proves the result for the bottom slot.
We briefly outline the proof for the top slot, it is simpler to work directly with $F_{A_1\dots A_k}=k\mu_{a_2 \dots a_{k}}Y_{[A_1}Z^{a_2}_{A_2}\cdots Z^{a_k}_{A_k]}$,
$$\begin{aligned} F_{A_1\dots A_k} &= \frac{1}{(k-1)!}\sum_{\sigma \in \mathfrak{S}_k}\varepsilon(\sigma)\mu_{a_2\dots a_k}Y_{A_{\sigma(1)}}Z^{a_2}_{A_{\sigma(2)}}\cdots Z^{a_k}_{A_{\sigma(k)}} 
\\&= \sum_{i=1}^k (-1)^{i-1}\mu_{a_1\dots a_{i-1}a_{i+1}\dots a_k}Y_{A_i}Z_{A_1}^{a_1}\cdots Z_{A_{i-1}}^{a_{i-1}}Z_{A_{i+1}}^{a_{i+1}}\cdots Z_{A_k}^{a_k}. \end{aligned}$$
In this case:
\[\begin{split}T\!=\!\!\!\!\sum_{\underset{\sigma \in \mathfrak{S}_{k\!+\!1}}{i=2}}^k\!\!\bigg((-1)^{i-1}\varepsilon(\sigma)W^A_a W_{d_1b\phantom{c}d_2}^{\phantom{d_1b}c}\mu_{cd_3\dots d_{i}d_{i\!+\!2}\dots d_k} Z^{d_1}_{D_{\sigma(1)}}\!\!Z^{d_2}_{D_{\sigma(2)}}\!\!Z_{D_{\sigma(3)}}^{d_3}\cdots Z_{D_{\sigma(i)}}^{d_i}\\\times Y_{D_{\sigma(i\!+\!1)}}\!\!Z_{D_{\sigma(i\!+\!2)}}^{d_{i\!+\!2}}\cdots Z_{D_{\sigma(k\!+\!1)}}^{d_{k+1}}\bigg).\end{split}\]
Note that the index $i$ starts at $2$ since the first term vanishes as: \[Y_CW^C_c=0.\]
We then apply the same method of computation as before to transform the sum over $\mathfrak{S}_{k+1}.$
\end{proof}

\subsection{Operator algebra}
The projectively invariant operators $\mathscr{D}, \mathscr{D}^*$ and $\sigma$ are the beginnings of an operator algebra that we will seek to exploit to write down a tractor version of the Proca equation. The commutators $[\mathscr{D},\sigma]$ and $[\mathscr{D}^*,\sigma]$ are directly related to the weight $1$ tractor: $I_A=D_A\sigma$, as follows:
\begin{lemm}
Define the operators: $\mathscr{I}: \mathcal{E}_{[A_1,\dots A_k]}(\omega)\rightarrow \mathcal{E}_{[A_1,\dots,A_{k+1}]}(\omega+1) $ and $\mathscr{I}^*: \mathcal{E}_{[A_1,\dots,A_k]}(\omega) \rightarrow \mathcal{E}_{[A_1,\dots, A_{k-1}]}(\omega+1)$ by:
$$\begin{gathered}\mathscr{I}F = I \wedge F$, where $I_A=D_A\sigma, \\ \mathscr{I}^* = \varepsilon s (-1)^{(k+1)(n+1)+1} \star \mathscr{I} \star.\end{gathered}$$
Then, in the scale $\nabla_g$ on $M$, for $F_{A_1\dots A_k}\overset{\nabla_g}{=}\begin{pmatrix} \mu \\ \xi \end{pmatrix} \in \mathcal{E}_{[A_1,\dots,A_k]}(\omega)$;
$$\mathscr{I}\begin{pmatrix} \mu \\ \xi \end{pmatrix}=\begin{pmatrix}2\sigma \xi \\ 0 \end{pmatrix} \textrm{ and } \mathscr{I}^*\begin{pmatrix}\mu \\\xi\end{pmatrix}=\begin{pmatrix}0 \\ -\frac{f}{2}\mu \end{pmatrix}=-I \lrcorner F.$$
Furthermore:
\begin{equation}\begin{gathered} \{\mathscr{I},\mathscr{I}^*\}=-f\sigma, \\ [\mathscr{D},\sigma]=\mathscr{I},  \quad [\mathscr{D}^*,\sigma]=\mathscr{I}^*, \\ \mathscr{I}^2={\mathscr{I}^*}^2=0. \end{gathered}\end{equation}
\end{lemm}
$\frac{1}{f}\mathscr{I}$ and $\frac{1}{f}\mathscr{I}^*$ play an analogous role to $\frac{1}{f}\mathscr{D}, \frac{1}{f}\mathscr{D}^*$ with respect $\tilde{y}$.
In the case of \emph{normal} solutions to the Metrisability equation, we can push things a little further:
\begin{lemm}
In the case that $H^{AB}$ is a normal solution to the metrisability equation then:
\begin{equation} \label{eq:anti_comm_D_scale} \{\mathscr{D}^*, \mathscr{I}\} = -\frac{f}{2}(\bm{\omega}+n+1-\bm{k}), \quad \{\mathscr{D}, \mathscr{I}^*\}=-\frac{f}{2}(\bm{\omega} +\bm{k}),\end{equation}
where $\bm{\omega}$ and $\bm{k}$ are respectively the weight and degree operators.
Furthermore:
$$\{\mathscr{D}^*, \mathscr{I}^*\} + \{\mathscr{D},\mathscr{I}^*\} =  -f h,$$
with: $h=\bm{\omega}+\frac{n+1}{2}$.
\end{lemm}
In particular, we have the following statement that generalises Lemma~\ref{lemme:boundary_calculus} to forms.
\begin{coro} \label{cor:boundary_calculus_forms}Suppose that $H^{AB}$ is a normal solution to the metrisability equation and set: $$\begin{cases} x=\sigma, \\ \tilde{y}= \frac{1}{f} \left( \mathscr{D}\mathscr{D^*}+\mathscr{D^*}\mathscr{D}\right), \\ h=\bm{\omega}+\frac{n+1}{2},\end{cases}$$ then $(x,\tilde{y},h)$ is an $\mathfrak{sl}_2$-triple.
\end{coro}
\begin{proof} $x$ increases weight by $2$ and $\tilde{y}$ decreases weight by $2$, hence: $[h,x]=2x$ and $[h,\tilde{y}]=-2\tilde{y}$, lastly, using the above results:
$$\begin{aligned} [x,\tilde{y}]&=\frac{1}{f}\left( [\sigma, \mathscr{D}\mathscr{D}^*] + [\sigma,\mathscr{D}^*\mathscr{D}]\right),\\
&= \frac{1}{f}\left([\sigma,\mathscr{D}]\mathscr{D}^* +\mathscr{D}[\sigma,\mathscr{D}^*] + [\sigma,\mathscr{D}^*]\mathscr{D} + \mathscr{D}^*[\sigma,\mathscr{D}] \right),\\
&=-\frac{1}{f}\left(\{\mathscr{I},\mathscr{D}^*\} + \{\mathscr{D},\mathscr{I}^*\} \right)\\&= h.\end{aligned}$$
\end{proof}
\begin{rema}
The commutator also follows from the Weitzenbock identity in Proposition~\ref{prop:weitzenbock} and S. Porath's $\mathfrak{sl}_2$ in Proposition~\ref{lemme:boundary_calculus}.
\end{rema}

\section{Asymptotic analysis of the Proca equation}\label{sec:asymptotics}

We have now developed enough tools in order to write down a Maxwell type system for general $k$-cotractor forms.
\[\begin{cases} \mathscr{D}F=0, \\ \mathscr{D}^*F=0.\end{cases}\]
However, we have not yet reaped all the benefits of Equation~\eqref{eq:anti_comm_D_scale}, which, in fact, contains important information on the cohomology spaces of the co-chain complex defined by $\mathscr{D}$. Note first that: $[\mathscr{D},\bm{\omega}+\bm{k}]=0=[\mathscr{I}^*,\bm{\omega} +\bm{k}]$. 
Therefore, if $\omega + k \neq 0$ and $F\in \mathcal{E}_{[A_1,\dots,a_{k}]}(\omega)$ satisfies $\mathscr{D}F=0$, then, according to Equation~\eqref{eq:anti_comm_D_scale}:
$$\mathscr{D}\left(-\frac{2}{f(\bm{\omega}+\bm{k})}\mathscr{I}^*F\right)=F.$$
In other words: 
\begin{prop}
Let $\omega \neq 0$,  then the cohomology spaces of the following co-chain complex are trivial:
$$ \Gamma(\mathcal{E}(\omega)) \overset{\mathscr{D}=D_{A_1}}\longrightarrow \mathcal{E}_{A_1}(\omega-1)  \overset{\mathscr{D}}\longrightarrow \dots  \overset{\mathscr{D}}\longrightarrow \mathcal{E}_{[A_1\dots A_{n+1}]}(\omega-(n+1)).$$  
\end{prop}
\begin{proof}
The case $k>0$ has already been treated. The case $k=0$ is easily seen as follows. In any scale $\nabla$ in the projective class: 
$$0=D_A f \overset{\nabla}{=} \begin{pmatrix}\omega f \\ \nabla_a f \end{pmatrix} \Rightarrow f=0,$$
because $\omega \neq 0$.
\end{proof}
The above Proposition simply means that as long as $\omega \neq - k$ there is always a tractor potential ! Thus, in this case, $\mathscr{D}F =0 \Leftrightarrow F=\mathscr{D}A$ and the co-tractor Maxwell system is completely equivalent to:
\begin{equation}
\label{eq:maxwell_proca}
\begin{cases}
F=\mathscr{D}A, \\ \mathscr{D}^*\mathscr{D}A =0.
\end{cases}
\end{equation}
The potential formulation has a manifest gauge symmetry and if one works in a Lorenz type gauge: $\mathscr{D}^*A=0$, the second equation becomes:
$$\tilde{y} A = 0.$$
Let us study what the equations $\mathscr{D^*}\mathscr{D}A=0$ and $\mathscr{D}^*A=0$ mean for the components of $A$ \emph{in the Levi-Civita} scale. 
Given our hypotheses, from Equation~\eqref{eq:codiff} we see that the gauge condition is: 
$$\mathscr{D}^*A \overset{\nabla_g}{=} \begin{pmatrix} -\delta \mu \\ \delta \xi -(\omega+n+1-k)\frac{f\sigma^{-1}}{4}\mu \end{pmatrix}=0,$$
from which we deduce that: $$\mu = \frac{4\sigma}{f(\omega+n+1-k)}\delta \xi,$$
provided that $\omega+n+1 -k\neq 0.$
Moreover, since in the scale $\nabla_g$, $\mathscr{D}^*\mathscr{D}\begin{pmatrix} \mu \\ \xi \end{pmatrix}$ is:
\begin{equation}
\begin{pmatrix} \delta \dd\mu - (\omega +k)\delta \xi \\ \delta \dd\xi +\frac{f\sigma^{-1}}{4}(\omega-1+n-k) \dd \mu -\frac{f\sigma^{-1}}{4}(\omega-1+n-k)(\omega+k)\xi\end{pmatrix},
\end{equation}
we see that the component $\xi$ satisfies a Proca equation with source where the mass is defined by:
\[m^2=(\omega-1+n-k)(\omega+k).\]

Now, let $\phi_{a_1\dots a_k}$ be a $k$-form on $M$. We can construct a weight $(\omega +k)$,  $k$-form in a natural manner by setting: 
$$\xi_{a_1\dots a_k} = \phi_{a_1\dots a_k}\sigma^{\frac{\omega+k}{2}}.$$
$\xi$ is then easily transformed into a weight $\omega$ co-tractor $k$-form via the map: 
$$\xi_{a_1\dots a_k} \longmapsto \xi_{a_1\dots a_k} Z_{A_1}^{a_1}\dots Z_{A_k}^{a_k}.$$
Setting $A=\xi_{a_1\dots a_k} Z_{A_1}^{a_1}\dots Z_{A_k}^{a_k}$ we see that the equation $\mathscr{D}^*\mathscr{D} A =0$ expressed in the Levi-Civita scale implies the gauge condition $\mathscr{D}^*A=0$ and implements on $\xi$ the Proca equation with mass defined above in the Lorenz gauge. The tractor formalism we have developed can therefore be used to study the asymptotics of $\xi$, through the study of $A$ and the equation $\tilde{y} A= 0$.

\subsection{Exploiting the $\mathfrak{sl}_2$ algebra: Formal solution operators} \label{sec:operateur_solution_formelle}
We shall now proceed to explain the relevance of Propositions~\ref{lemme:boundary_calculus} and Corollary~\ref{cor:boundary_calculus_forms}.
As stated previously, they can be used to construct coordinate invariant \emph{formal} solutions operators to the problem: $y\tau=0$, where $y$ here can be either $H^{AB}D_AD_B$ or $\frac{1}{f}\{\mathscr{D},\mathscr{D}^*\}$ under the appropriate assumptions. The method is identical to that in~\cite{Gover:2014aa} in the conformal case. The idea is to look for \emph{formal} operators $A$, generated by $x^\alpha, \alpha \in \mathbb{C}$ ($x=\sigma$) and $y$, that annihilate $y$ from the right, i.e. that satisfy $yA=0$.
Inspired by the Frobenius method, one can seek solutions of the form:
\[A = x^{\nu}\sum_{k=0}^\infty \alpha_k x^ky^k. \]
Now, note the same reasoning outlined in the proof of Lemma~\ref{lemme:commutator_laplace_sigma}, can be used to prove that for any complex $\nu \in \mathbb{C}$,
\begin{equation} [x^\nu,y]=x^{\nu-1}\nu(h+\nu -1).\end{equation}
Hence, formally~:
\[\begin{aligned} yA&=yx^{\nu}\sum_{k=0}^{\infty}\alpha_k x^ky^k\\&= x^{\nu}\sum_{k=1}^{\infty}\alpha_k yx^ky^k - x^{\nu-1}\sum_{k=0}^{\infty}\nu(h+\nu-1)\alpha_k x^ky^k,&\\&=x^{\nu}\sum_{k=0}^{\infty}\alpha_k x^ky^{k+1}-\sum_{k=0}^\infty x^{k-1} k(h+k-1)\alpha_k y^{k} \\&\hspace{6cm}- x^{\nu-1}\sum_{k=0}^{\infty}\nu(h+\nu-1)\alpha_k x^ky^k.\end{aligned}\]
Considering the action of $A$ on an eigenspace of the operator $h=\bm{\omega}+\frac{d+2}{2}$ with a fixed eigenvalue $h_0$, we have~:
\[\begin{aligned} yA&= x^{\nu}\sum_{k=0}^{\infty}\alpha_k x^ky^{k+1}-\sum_{k=0}^\infty x^{k-1} k(h_0-k-1)\alpha_k y^{k} \\ &\hspace{6cm}- x^{\nu-1}\sum_{k=0}^{\infty}\nu(h_0+\nu-1)\alpha_k x^ky^k, \\&= x^{\nu-1} \left(\sum_{k=1}^\infty\alpha_{k-1} x^ky^k-(h_0-1)\sum_{k=0}^\infty k \alpha_k x^{k}y^{k}\right. \\ &\hspace{1.5in}\left. + \sum_{k=0}^\infty k^2\alpha_kx^{k}y^{k}  -\nu(h_0+\nu-1)\sum_{k=0}^\infty \alpha_{k} x^{k}y^{k}\right). \end{aligned}\]
In order to ensure $Ay=0$, we demand that the operator between brackets vanish identically.
To find a solution, it is necessary to be a little more precise about how we would like $A$ to act. In fact, the idea would be to take some smooth data $f_0$ on the boundary, extend it arbitrarily to $\bar{f}_0$ over $M$ and $A\bar{f}_0$ should satisfy $yA\bar{f}_0=0$ and $x^{-\nu}A\bar{f}_0$ should restrict to $f_0$ on the boundary. In other words, $\alpha_0 =1$. Thus, after rewriting the above equation in terms of a formal series $\displaystyle F(z)~:=\sum_{k=0}^\infty \alpha_k z^k \in \mathbb{C}[[z]]$, where $z^k=:(xy)^k:=x^ky^k$, we see that necessarily: 
\begin{equation}\label{eq:restrict_formal_series} \nu(h_0+\nu-1)=0.\end{equation}
Taking this into account, we find that the formal series $F$ satisfies the ODE:
\[(zF')' -(h_0-1)F' +F=0.\]
Equation~\eqref{eq:restrict_formal_series} should be compared with the indicial equation~\eqref{eq:indicial_frob} we will obtain when applying the Frobenius method in de-Sitter space. It is interesting to note that the coefficients $\alpha_k$ are completely independent of the geometry and identical to those in the conformal case.

\subsection{Asymptotics of solutions to the Klein-Gordon equation in de-Sitter spacetime}
 We shall now consider the specific case $(1+d)$-dimensional de-Sitter spacetime and return to the notations introduced in Section~\ref{sec:compactification_de_sitter}. Let $\sigma\in \mathcal{E}(2)$ be the defining density for the boundary constructed from the volume form $\omega_g$ by $\sigma =|\omega_g|^{-\frac{2}{d+2}}$ and $\zeta^{ab}=\sigma^{-1}g^{ab}$. We shall argue that in this case, the formal solution operators above implement exactly the Frobenius method, after reducing the equation to a $1$-dimensional operator thanks to the numerous symmetries of de-Sitter spacetime. For simplicity we shall consider scalar fields, in which case, the operators $H^{AB}D_AD_B$ and $\{\mathscr{D}, \mathscr{D}^*\}$ coïncide are essentially the same. We will therefore work with $y=\frac{-1}{I^2}H^{AB}D_AD_B$.
 
 In the scale defined by $\sigma$~:
 \[\zeta^{ab}P_{ab}=\frac{1}{d}\zeta^{ab}R_{ab}=\sigma^{-1}(d+1)\] so Equation~\eqref{eq:laplac} becomes~:
\begin{equation} \label{eq:laplac2} \Delta^\mathcal{T}\tau \overset{\nabla_\zeta}{=} \sigma^{-1}\left(\omega(\omega+d)\tau + g^{ab}\nabla_a\nabla_b \tau\right). \end{equation}
This is clearly related to the Klein-Gordon operator with mass defined by the relation $\omega(\omega+d)=-m^2$. Vice versa, for a given value of $m$ there are therefore two weights on which $\Delta^{\mathcal{T}}$ acts exactly as the Klein-Gordon operator with mass $m$~:
\[\omega_m \in\left\{ \frac{1}{2}\left(-d + \xi \right), \xi^2 =d^2-4m^2\right\},\] generically, $\xi$ is complex.\newline
\indent On de-Sitter spacetime, the operator $y$ is simply $y=-\Delta^{\mathcal{T}}$, therefore the equation $y\tau=0$ for $\tau \in \mathcal{E}(\omega_m)$, in the scale determined by $\sigma$, is the Klein-Gordon equation for a classical scalar field with mass $m$. More precisely, solutions to the Klein-Gordon equation with mass $m$ on de-Sitter space are in one-to-one correspondence with densities of weight $\omega_m$ in $\ker y$; the correspondence being accomplished naturally via the map $\phi \mapsto \phi\sigma^{\frac{\omega_m}{2}}\equiv \tau_\phi$. 

We now look at the expression of $y=\frac{-1}{I^2}H^{AB}D_AD_B$ expressed in a scale that is regular at the boundary.  Introduce the coordinate functions $(\psi, \vartheta), \psi \in \mathbb{R}, \vartheta \in S^{n}$ and recall that $\rho = \frac{1}{2\cosh^2\psi}$ is a boundary defining function. Each local frame $(\omega^0,\dots,\omega^d)$ on $T^*M$ defines a positive density of any weight $\omega$ that we will call $|\omega^0\wedge\dots \wedge \omega^d|^{-\frac{\omega}{d+2}}$. For simplicity: let $\omega^1, \dots, \omega^d$ be dual to an orthonormal frame on $TS^d$
and write $\omega^1\wedge \dots \wedge \omega^d = \textrm{d}\Omega^d.$  Then $|\textrm{d}\rho \wedge d\Omega^d|^\frac{-\omega}{d+2}$ is smooth up to the boundary and:
\[\sigma = 2\rho(1-2\rho)^{\frac{1}{d+2}}|\textrm{d}\rho \wedge \textrm{d}\Omega^n|^\frac{-2}{d+2}.\]
By construction, the connection $\nabla^s=\nabla_g + \frac{\textrm{d}\rho}{2\rho}$ extends to the boundary and preserves the 2-density $s = \frac{1}{\rho}\sigma$; the scale $s$ can therefore be used to study $y$ near the boundary $\sigma =0$.  Using the change of connection formulae, we find:
\[H^{AB}\overset{\nabla_s}= \begin{pmatrix} s^{-1}\rho^{-1} g^{ab} \\ 2s^{-1}(1-2\rho)\partial^a_\rho \\ 2s^{-1} \end{pmatrix}. \]
Hence, if expressed in terms of the connection $\nabla^s$, for $\tau\in \mathcal{E}(\omega)$~:
\begin{equation} y\tau= 2s^{-1}(\omega-1)\omega \tau +4s^{-1}\rho(1-2\rho)(\omega-1)\partial^a_\rho{\nabla^s}_a \tau + \zeta^{ab}\left(\nabla^s_a\nabla^s_b\tau +\omega P^s_{ab}\tau \right).  \end{equation}
Writing $\tau=\phi s^{\frac{\omega}{2}}$, and using the fact that $\nabla^s s=0$,  Lemma~\ref{lemme:box_connexion_chapeau}, shows that:
\begin{equation} \begin{aligned} &y\tau 
= s^{\frac{\omega}{2}-1}\rho^{-1}\Box_s \phi +4s^{-1}(\omega-1)(1-2\rho)\partial_\rho \phi \\& \hspace{2.5in}+2s^{-1}(\omega+d-1)\omega \phi ), 
\\&= 2s^{\frac{\omega}{2}-1}\bigg(-2\rho(1-2\rho)\partial^2_\rho \phi +\Delta_{S^d}\phi +\omega(\omega+d-1)\phi \\ &\hspace{1in}+[2(\omega-1)(1-2\rho)+(2\rho(1-d)+d)]\partial_\rho \phi \bigg).\end{aligned}\end{equation}
Therefore, near the boundary, $y\tau=0, \tau=\phi s^\frac{\omega}{2}$ is equivalent to:
\[-2\rho(1-2\rho)\partial^2_\rho \phi + (1+(1-2\rho)(d-3+2\omega))\partial_\rho \phi + \Delta_{S^d}\phi +\omega(\omega+d-1)\phi=0.\]
Exploiting the spherical symmetry of the above equation by decomposing onto a spherical harmonic $\lambda=l(l+d-1)$, the problem is reduced to the ODE:
\begin{equation}\label{eq:odeds}\begin{split} -2\rho(1-2\rho)\partial^2_\rho \phi + (1+(1-2\rho)(d-3+2\omega))\partial_\rho \phi \\+(\omega(\omega+d-1)-\lambda)\phi=0.\end{split}\end{equation}
Since all coefficients are analytic functions of $\rho$, it is well adapted to the Frobenius method~\cite{Ince:1956aa} and so we seek solutions of the form~:
\[\phi= \rho^\nu \sum_{k\geq 0} \alpha_k \rho^k.\]
Plugging this ansatz into Equation~\eqref{eq:odeds} yields the indicial equation~: \begin{equation}\label{eq:indicial_frob} \nu(2\omega + n -2\nu)= 2\nu(h_0-\nu-1)=0,\end{equation}
where we have introduced $h_0= h\tau=\omega + \frac{d+2}{2} $ and $h$ is the operator defined in Lemma~\ref{lemme:boundary_calculus}.
Hence: \[\nu=0 \textrm{ or } \nu = h_0 -1.\]
For $k\geq 1$, the coefficients $\alpha_k$ satisfy the following recurrence relation~:
\[2(\nu +k)(h_0 -k-1-\nu))\alpha_k =c_k\alpha_{k-1},\]
with:
\[  c_k= 2(\nu + k-1)(2\nu +k+1-2h_0) + \omega(\omega +n-1) - \lambda.\]
This is readily solved for any given $\alpha_0$ provided that for all $k \in \mathbb{N}$, $k \neq h_0-1$ (when $\nu=0$)  or $k\neq  -(h_0-1)$ when ($\nu=h_0-1$).
In a generic case $h_0 \in \mathbb{C}\setminus \mathbb{R}$, and there is no obstruction to the existence of the series. Under the assumption that we avoid these special cases, the Frobenius method yields two independent solutions to the equation, and generic smooth solutions can be written~:
\[\phi = \phi_0 + \rho^{h_0-1}\phi_1,\]
where $\phi_0, \phi_1$ are regular up to the boundary.
Now, returning to the scale $\nabla_g$, \[\tau=\tilde{\phi}\sigma^{\frac{\omega}{2}}=\tilde{\phi}\rho^{\frac{\omega}{2}}s^\frac{\omega}{2}.\]
Hence: \begin{equation} \label{eq:asymp_desitter} \tilde{\phi} = \phi_0\rho^{-\frac{\omega}{2}} + \rho^{h_0-\frac{\omega}{2}-1}\phi_1=\phi_0\rho^{-\frac{\omega}{2}} + \rho^{\frac{\omega}{2}+\frac{d}{2}}\phi_1.\end{equation}
Choosing $\omega \in \{ \frac{1}{2}(-d+\xi), \xi^2 =d^2 - 4m^2\}$, Equation~\eqref{eq:asymp_desitter} describes the asymptotic behaviour of solutions to the Klein-Gordon equation near the projective boundary. Observe from~\eqref{eq:asymp_desitter} that the precise choice of weight in the identification $\phi \mapsto \phi \sigma^{\frac{\omega}{2}}$ is inconsequential and switching between the two possible values at fixed mass $m$ amounts to exchanging $\phi_0$ and $\phi_1$. Overall, solutions behave asymptotically as~:
\[\tilde{\phi} = \phi_0\rho^{\frac{1}{4}(d-\sqrt{d^2-4m^2})} + \phi_1\rho^{\frac{1}{4}(d+\sqrt{d^2-4m^2})}.\]
Where $\sqrt{d^2-4m^2}$ is a (perhaps complex) square root of $d^2-4m^2$.
This result should be compared with~\cite[Theorem 1.1]{VASY201049}. 
The important point is that the powers appearing in the above expansion are the same as those predicted by the formal series.

\section{Conclusion}
In this article, we have established, on a class of projectively compact manifolds, results that are parallel to those available in the case of conformally compact manifolds. In particular, we have constructed an exterior tractor calculus on order 2 projectively compact manifolds. It is hoped that this will constitute a basis for a geometric approach to the asymptotic analysis of classical fields on such backgrounds that would be an alternative to microlocal analysis. There are still some outstanding questions that we have not been able to touch upon. In particular, it is not yet clear how to give a clear-cut analytical meaning to the formal solution operators we obtain (as the symbol of a Fourier integral operator, for instance), and the question of how to treat the asymptotically flat case, in which the structure at the basis of the formal construction becomes trivial, remains open. This will be the object of work in the near future.

\appendix
\section{Connection forms in Minkowski spacetime in coordinates that extend to the boundary}
\label{annexe:minkowski}
First, those of the Levi-Civita connection are :
 \begin{equation}
\begin{gathered}
\omega^0_{\,\,0} = -2\frac{\textrm{d}\rho}{\rho}; \quad \omega^{i}_{\,\,0}=-\frac{\textrm{d}\tilde{x}_i}{\rho}, \\ \omega^0_{\,\, j} = -\rho \textrm{d}\tilde{x}_j + \frac{\rho \tilde{x}_j}{1+|\tilde{x}|^2}\sum_k \tilde{x}_k \textrm{d}\tilde{x}_k=\rho^3\sum_{k}g_{jk}\textrm{d}\tilde{x}_k, \\
\omega^{i}_{\,\, j} = -\tilde{x}_i\textrm{d}\tilde{x}_j + \frac{\tilde{x}_i\tilde{x}_j}{1+|\tilde{x}|^2} \sum_k \tilde{x}_k\textrm{d}\tilde{x}_k-\delta_{ij}\frac{\textrm{d}\rho}{\rho}=\tilde{x}_i\rho^2\sum_{k}g_{jk}\textrm{d}\tilde{x}_k-\delta_{ij}\frac{\textrm{d}\rho}{\rho}.
\end{gathered}
\end{equation}
\noindent Those of the connection $\hat{\nabla}=\nabla + \frac{\textrm{d}\rho}{\rho}$ are obtained by applying Equation~\eqref{eq:proj_equiv_local}:
 \begin{equation} \begin{gathered} \hat{\omega}^{0}_{\,\,0}=\hat{\omega}^{i}_{\,\,0}=0, \\
\hat{\omega}^{0}_{\,\, j} = \omega^{0}_{\,\, j}, \\ \hat{\omega}^{i}_{\,\,j} = \tilde{x}_i\rho^2\sum_{k}g_{jk}\textrm{d}\tilde{x}_k, \end{gathered} \end{equation}
and clearly have smooth extensions to points where $\rho$ vanishes, which proves projective compactness of order 1.

\section{Useful formulae in de-Sitter spacetime}
\label{annexe:desitter}
In appropriate coordinates $(1+d)$-dimensional de-Sitter space $(dS,g)$ is the \enquote{warped} direct product of the pseudo-Riemannian manifolds $(\mathbb{R}_\psi, -d\psi^2)$ and $(S^d, d\sigma^d)$, where $d\sigma^d$ is the usual round metric in $d$-dimensions. In other words, $dS= \mathbb{R} \times S^d$ but the metric is given by : $$g=-d\psi^2 +f(\psi)d\sigma^d, f(\psi)=\cosh^2\psi.$$
The function $f$ is responsible for the \enquote{warping} of the direct product. In well-chosen local frames of $dS^d$ that are adapted to the direct sum decomposition of $dS$ into $\mathbb{R}\times S^d$, it is possible to determine the local connection forms in terms of those of $d\sigma^d$ on $S^d$ and $-d\psi^2$ on $\mathbb{R}$ .

For our purposes we will work on a coordinate patch where $\psi$ can be replaced by the boundary defining function $\rho=\frac{1}{2\cosh^2\psi}$. We recall that $g$ is then given by :
\[g=-\frac{\textrm{d} \rho^2}{4\rho^2(1-2\rho)} + \frac{1}{2\rho} \textrm{d}\sigma^d.\]
Choosing a local orthonormal frame on $S^d$ and writing $\theta^{i}_j$ for the local connection forms on $S^d$ in this basis then the matrix-valued local connection form for the Levi-Civita connection of $g$ is:
\begin{lemm}
\begin{equation} (\omega^i_{\,\,j})_{1\leq i,j \leq n+1} = \begin{pmatrix} \left(\frac{1}{1-2\rho} - \frac{1}{\rho} \right)\textrm{d}\rho & -(1-2\rho)\omega^j_\theta \\ -\frac{1}{2\rho}\omega^i_\theta & \theta^{i}_j -\frac{\textrm{d}\rho}{2\rho} \end{pmatrix}. \end{equation} 
\end{lemm}
From this it follows that :
\begin{equation} \Box \rho = g^{ab}\nabla_a\nabla_b \rho = 2\rho(d-2 +2\rho(3-d)), \end{equation}
and the local connection form matrix in the same local-frame corresponding to the connection $\hat{\nabla} = \nabla + \frac{\textrm{d}\rho}{2\rho}$ is:
\begin{lemm}
\begin{equation} (\hat{\omega}^i_{\,\,j})_{1\leq i,j \leq n+1}\begin{pmatrix} \frac{\textrm{d}\rho}{1-2\rho}& -(1-2\rho)\omega^j_\theta \\0 & \theta^{i}_j\end{pmatrix}. \end{equation} 
\end{lemm}
Using this one can show by a direct computation that, acting on scalar fields:
\begin{lemm}
\label{lemme:box_connexion_chapeau}
\[g^{ab}\hat{\nabla}_a\hat{\nabla}_b  = -4\rho^2(1-2\rho)\partial^2_\rho +2\rho(2\rho(1-d)+d)\partial_\rho + 2\rho\Delta_{S^n}.\]
\end{lemm}

\section{Proof of Lemma~\ref{lemme:tractorformcouple}}
\label{app:preuve_lemme_tractorformcouple}
Let us prove the transformation law~\eqref{eq:tractor_form_transform} for the components of a $k$-cotractor form when we change connection according to $\hat{\nabla}=\nabla +\Upsilon$. This can be done by induction. The case $k=1$ is well known, but we prove it here for completeness. Consider: 
\[F_{A}\overset{\nabla}=\begin{pmatrix} \sigma \\ \mu_a \end{pmatrix} \overset{\hat{\nabla}}=\begin{pmatrix} \hat{\sigma} \\ \hat{\mu}_a \end{pmatrix}\] then for any tractor $T^A\overset{\nabla}=\rho Y_A + \nu^bW^A_b\overset{\hat{\nabla}}= \hat{\rho} Y_A + \hat{\nu}^bW_A^b$:
\[F_AT^A=\rho\sigma +\mu_b\nu^b=\hat{\rho}\hat{\sigma}+\hat{\mu}_b\hat{\nu}^b.\]
Using Equation~\eqref{eq:tractor_transfo}, it follows that: 
\[\rho\sigma + \mu_b\nu^b = (\rho-\Upsilon_b\nu^b)\hat{\sigma}+\hat{\mu}_b\nu^b.\]
Hence: \[\rho(\sigma-\hat{\sigma}) +\nu^b(\mu_b+\Upsilon_b\hat{\sigma}-\hat{\mu}_b)=0,\]
since this holds for arbitrary $\rho, \nu^b$ we conclude that:
\[\left\{ \begin{array}{l} \hat{\sigma} =\sigma, \\ \hat{\mu}_b=\mu_b+\Upsilon_b \sigma.\end{array}\right.\]
Assume now that~\eqref{eq:tractor_form_transform} is true for $k$-forms, we prove that it then holds for $k+1$ forms.
Let $F_{A_1\dots A_{k+1}}\overset{\nabla}=(k+1)\mu_{a_2\dots a_{k+1}}Y_{[A_1}Z^{a_2}_{A_2}\cdots Z^{a_{k+1}}_{A_{k+1}]}+\xi_{a_1\dots a_{k+1}}Z^{a_1}_{A_1}\cdots Z^{a_{k+1}}_{A_{k+1}}$. As in the case $k=1$, let $T^A\overset{\nabla}=\rho X^A + \nu^bW^A_b$ be an arbitrary tractor, we calculate the contraction:
\[\begin{aligned}F_{A_1\dots A_{k+1}}T^{A_{k+1}}&= ((-1)^{k}\rho\mu_{a_1\dots a_{k}}+\xi_{a_1\dots a_kb}\nu^{b})Z_{A_1}^{a_1}\dots Z_{A_{k}}^{a_{k}}\\&\hspace{2.3cm} +(k+1)\mu_{a_2\dots a_{k+1}}Y_{[A_1}Z^{a_2}_{A_2}\cdots Z^{a_{k+1}}_{A_{k+1}]}\nu^bW_b^{A_{k+1}}.\end{aligned}\]
\noindent The final term requires special attention:
\[\begin{aligned}(k+1)\mu_{a_2\dots a_{k+1}}Y_{[A_1}Z^{a_2}_{A_2}\cdots Z^{a_{k+1}}_{A_{k+1}]}\nu^bW_b^{A_{k+1}}=\hspace{2in}\\\frac{1}{k!}\sum_{\sigma\in\mathfrak{S}_{k+1}}\!\!\!\!\varepsilon(\sigma)\mu_{a_2\dots a_{k+1}} Y_{A_\sigma(1)}Z^{a_2}_{A_{\sigma(2)}}\cdots Z^{a_{k+1}}_{A_{\sigma(k+1)}}W^{A_{k+1}}_b\nu^b.\end{aligned}\]
If $\sigma(1)=k+1$, then the summand is zero, hence:
\[\begin{aligned}\sum_{\sigma \in \mathfrak{S}_k}\varepsilon(\sigma)\mu_{a_2\dots a_{k+1}} Y_{A_{\sigma(1)}}Z^{a_2}_{A_{\sigma(2)}}\cdots Z^{a_{k+1}}_{A_{\sigma(k+1)}}W^{A_{k+1}}_b\nu^b= \hspace{1.4in}\\\ \sum_{i=1}^{k} \sum_{\underset{\sigma(1)=i}{\sigma \in \mathfrak{S}_{k+1}}}\varepsilon(\sigma)\mu_{a_2\dots a_{k+1}} Y_{A_i}Z^{a_2}_{A_{\sigma(2)}}\cdots Z^{a_{k+1}}_{A_{\sigma(k+1)}}W^{A_{k+1}}_b\nu^b.\end{aligned}\]
Reorganising the terms in the product, we have that:
\[\begin{aligned} &\mu_{a_2\dots a_{k+1}} Y_{A_i}Z^{a_2}_{A_{\sigma(2)}}\cdots Z^{a_{k+1}}_{A_{\sigma(k+1)}}W^{A_{k+1}}_b\nu^b\hspace{2in}\\ &\hspace{.5in}=\mu_{a_2\dots a_{k+1}} Y_{A_i}Z^{a_{\sigma^{-1}(1)}}_{A_1}\cdots Z^{a_{\sigma^{-1}(i-1)}}_{A_{i-1}}Z^{a_{\sigma^{-1}(i+1)}}_{A_{i+1}} \cdots Z^{a_{\sigma^{-1}(k)}}_{A_{k+1}}W^{A_{k+1}}_b\nu^b, \\ &\hspace{.5in}= \mu_{a_{\sigma(2)}\dots {a_{\sigma(k+1)}}}\nu^{a_{k+1}}Y_{A_i}Z^{a_1}_{A_1}\dots Z^{a_{i-1}}_{A_{i-1}}Z^{a_{i+1}}_{A_{i+1}} \dots Z^{a_k}_{A_k}, \\&\hspace{.3in}= (-1)^{i-1}\varepsilon(\sigma)\mu_{a_1a_2\dots a_{i-1}a_{i+1}\dots a_{k+1}}\nu^{a_{k+1}}Y_{A_i}Z^{a_2}_{A_2}\dots Z^{a_{i-1}}_{A_{i-1}}Z^{a_{i+1}}_{A_{i+1}} \dots Z^{a_k}_{A_k}.\end{aligned}\]
The final equation comes from the following observation. If we relabel: \[\mu_{a_{\sigma(2)}\dots a_{\sigma(k+1)}}= \mu_{\bar{a}_1\dots \bar{a}_k};$$ then for any $s\in \mathfrak{S}_k$, $$\mu_{\bar{a}_{s(1)}\bar{a}_{s(2)}\dots \bar{a}_{s(k)}}=\mu_{a_{\sigma(s(1)+1)}\dots a_{\sigma(s(k)+1)}}.\]
Since $\sigma( \{2,\dots k\}) =\llbracket 1, k+1 \rrbracket \setminus\{ i\}$, we can reorder the indices such that we have $\mu_{a_1\dots a_{i-1}a_{i+1}\dots a_{k+1}}$ if we choose $s$ such that :
\[s(j) = \begin{cases}\sigma^{-1}(j)-1 & \textrm{if $1\leq j < i$}, \\ \sigma^{-1}(j+1) -1 & \textrm{if $i\leq j \leq k$}. \end{cases}\]
The signature of this permutation can be determined\footnote{in a rather tedious way} to be $(-1)^{i-1}\varepsilon(\sigma)$. One can observe that in the quotient group $\mathbb{R}^*/\mathbb{R}^*_+$:
\[\begin{aligned} \varepsilon(s) &=\prod_{1\leq m<l\leq k}\!\!\!\!s(l)-s(m),\\&=\prod_{1\leq m < l < i}\!\!\!\!\sigma^{-1}(l)-\sigma^{-1}(m)\!\!\!\!\prod_{1\leq m <i \leq l \leq k}\!\!\!\!\sigma^{-1}(l+1)-\sigma^{-1}(m)\\&\hspace{2.3in}\times\prod_{i\leq m<l \leq k}\!\!\!\!\sigma^{-1}(l+1)-\sigma^{-1}(m+1),\\&=\prod_{1\leq m < l < i}\!\!\!\!\sigma^{-1}(l)-\sigma^{-1}(m)\!\!\!\!\prod_{1\leq m <i < l \leq k+1}\!\!\!\!\sigma^{-1}(l)-\sigma^{-1}(m)\\&\hspace{2.3in}\times\prod_{i+1\leq m<l \leq k+1} \sigma^{-1}(l)-\sigma^{-1}(m). \end{aligned}\]
This differs from $\varepsilon(\sigma^{-1})$ by the sign of : \[ \prod_{1\leq m < i} 1-\sigma^{-1}(m) \equiv (-1)^{i-1} \textrm{mod } \mathbb{R}_+^*.\]
Overall, we find that :
\[\begin{split}F_{A_1\dots A_{k+1}}T^{A_{1}}= ((-1)^{k}\rho\mu_{a_1\dots a_{k}}+\nu^b\xi_{ba_1\dots a_{k}})Z_{A_1}^{a_1}\dots Z_{A_{k}}^{a_{k}}\hspace{.5in} \\+k\mu_{a_2a_3\dots a_ka_{k+1}}\nu^{a_{k+1}}Y_{[A_1}Z^{a_2}_{A_2}\cdots Z^{a_{k}}_{A_{k}]}.\end{split}\]
This is a $k$-cotractor, therefore, according to the induction hypothesis we must have:
\[\left\{ \begin{array}{l} \mu_{a_2\dots a_{k}b}\nu^{b}=\hat{\mu}_{a_2\dots a_k b}\hat{\nu}^{b},\\
(-1)^k\hat{\rho}\hat{\mu}_{\substack{\\a_1\dots a_{k}}}\!\!\!+\hat{\nu}^b\hat{\xi}_{a_1\dots a_{k}b}\!=\!(-1)^k\rho\mu_{\substack{\\a_1\dots a_{k}}}\!\!\!+\xi_{a_1\dots a_{k}b}\nu^b \!+\! k\Upsilon_{[a_1}\mu_{a_2\dots a_k]b}\nu^{b}. \end{array}\right.\]
Plugging the first equation into the second and using Equation~\eqref{eq:tractor_transfo}, we have 
\[\hat{\xi}_{a_1\dots a_{k}b}\nu^b= \nu^b\left(\xi_{a_2\dots a_{k}b} + \underbrace{k\Upsilon_{[a_1}\mu_{a_2\dots a_{k}]b} + (-1)^k\Upsilon_b\mu_{a_1\dots a_{k}}}_{k\Upsilon_{[a_1}\mu_{a_2\dots a_kb]}}\right).\]
The tractor $T^A$ being arbitrary, it follows that :
\[\left\{\begin{array}{l} \mu=\hat{\mu}, \\ \hat{\xi} = \xi + \Upsilon \wedge \mu, \end{array} \right. \]
and the result follows by induction.
\section{Hodge star of wedge product}
\begin{prop}
\label{prop:hodgestarwedge}
Let $\xi$ and $\Upsilon$ be respectively a $k$-form and a $1$-form on a pseudo-Riemannian manifold $(M,g)$, then :
\begin{equation} \label{eq:hodgestarwedge} \star(\Upsilon \wedge \xi) =(-1)^k\Upsilon^\sharp\lrcorner(\star \xi) \end{equation}
\end{prop}
\begin{proof}
The reason is essentially the fact that contraction and wedge product are adjoint operators. Precisely, if $\xi$ is $k$-form and $\alpha$ an arbitrary $(k+1)$-form then:
\begin{equation} \label{eq:wedgecontractionduality} g(\Upsilon\wedge \xi, \alpha)=g(\xi,\Upsilon^\sharp\lrcorner \alpha).\end{equation}
Postponing for now the proof of~\eqref{eq:wedgecontractionduality}, we prove Equation~\eqref{eq:hodgestarwedge}.
For any $k+1$-form $\alpha$:
\[\begin{aligned}\alpha \wedge \star(\Upsilon \wedge \xi) &= g(\alpha,\Upsilon\wedge\xi)\omega_g\\&=g(\Upsilon^\sharp\lrcorner\alpha,\xi)\omega_g\\&=(\Upsilon^\sharp \lrcorner \alpha) \wedge \star \xi. \end{aligned}\]
Since, $\Upsilon^\sharp \lrcorner \underbrace{(\alpha \wedge \star \xi)}_{=0}=(\Upsilon^\sharp \lrcorner \alpha)\wedge \star \xi - (-1)^{k}\alpha \wedge (\Upsilon^\sharp\lrcorner \star \xi)$
it follows that for any $\alpha$:
\[\alpha\wedge \left( (-1)^k \Upsilon^\sharp \lrcorner \star \xi \right) = g(\alpha, \star(\Upsilon \wedge \xi))\omega_g.\]
This property uniquely defines the Hodge star, therefore :
\[\star(\Upsilon \wedge \xi) =(-1)^k\Upsilon^\sharp\lrcorner \star \xi.\]
We prove now~\eqref{eq:wedgecontractionduality}, for instance, using the abstract index notation:
\[\begin{aligned} g(\Upsilon&\wedge \xi, \alpha)=\frac{1}{(k+1)!}g^{a_1b_1}\cdots g^{a_{k+1}b_{k+1}}(k+1)\Upsilon_{[a_1}\xi_{a_2\dots a_{k+1}]}\alpha_{b_1\dots b_{k+1}},\\ &=\frac{1}{k!(k+1)!}\sum_{\sigma \in \mathfrak{S}_{k+1}}\varepsilon(\sigma)g^{a_1b_1}\cdots g^{a_{k+1}b_{k+1}}\Upsilon_{a_{\sigma(1)}}\xi_{a_{\sigma(2)}\dots a_{\sigma(k+1)}}\alpha_{b_1\dots b_{k+1}},\\&=\frac{1}{k!(k+1)!}\sum_{\sigma \in \mathfrak{S}_{k+1}}\varepsilon(\sigma)g^{a_1b_{\sigma(1)}}\cdots g^{a_{k+1}b_{\sigma(k+1)}}\Upsilon_{a_{1}}\xi_{a_2\dots a_{k+1}}\alpha_{b_1\dots b_{k+1}},\\&=\frac{1}{k!(k+1)!}\sum_{\sigma \in \mathfrak{S}_{k+1}}\varepsilon(\sigma)g^{a_1b_1}\cdots g^{a_{k+1}b_{k+1}}\Upsilon_{a_{1}}\xi_{a_2\dots a_{k+1}}\alpha_{b_{\sigma^{-1}(1)}\dots b_{\sigma^{-1}(k+1)}},\\&= \frac{1}{k!(k+1)!}\sum_{\sigma \in \mathfrak{S}_{k+1}}g^{a_1b_1}\cdots g^{a_{k+1}b_{k+1}}\Upsilon_{a_{1}}\xi_{a_2\dots a_{k+1}}\alpha_{b_{1}\dots b_{k+1}},\\&=\frac{1}{k!}g^{a_2b_2}\cdots g^{a_{k+1}b_{k+1}}\xi_{a_2\dots a_{k+1}}g^{a_1b_1}\Upsilon_{a_{1}}\alpha_{b_{1}\dots b_{k+1}},\\&=g(\xi,\Upsilon^\sharp\lrcorner \alpha).\end{aligned}\]
\end{proof}

\bibliographystyle{smfplain}
\bibliography{biblio}
\end{document}